\documentclass[10pt]{amsart}
\usepackage{graphicx}
\usepackage{amssymb}
\usepackage{amsmath,amsthm}
\usepackage{hyperref}
\usepackage{amsthm}
\usepackage[english]{babel}
\usepackage{mathrsfs}

\newtheorem{thm}{Theorem}[section]
\newtheorem*{thm*}{Theorem}
\newtheorem{cor}[thm]{Corollary}
\newtheorem*{cor*}{Corollary}
\newtheorem{lem}[thm]{Lemma}
\newtheorem{prop}[thm]{Proposition}
\theoremstyle{definition}

\newtheorem{rem}[thm]{Remark}

\newtheorem{examp}[thm]{Example}

\newtheorem{conj}[thm]{Conjecture}
\newcommand{\car}{\curvearrowright}

\newcommand{\dd}{\delta}
\newcommand{\G}{\Gamma}
\newcommand{\g}{\gamma}
\newcommand{\e}{\varepsilon}
\def\La{\Lambda}
\def\b{\beta}
\def\ra{\rightarrow}
\def\le{\leqslant}
\def\ge{\geqslant}
\def\l{\lambda}

\newcommand{\N}{{\mathbb N}}
\newcommand{\Z}{{\mathbb Z}}

\newcommand{\HH}{{\mathcal H}}
\newcommand{\KK}{{\mathcal K}}

\newcommand{\NN}{{\mathcal N}}

\newcommand{\UU}{{\mathcal U}}

\newcommand{\oo}{{\overline \otimes}}

\newcommand{\id}{\operatorname{id}}

\title[Unique group-measure space decomposition]{Some unique group-measure space decomposition results}
\author{Ionut Chifan}
\address{Ionut Chifan, Vanderbilt University, 1326 Stevenson Center, Nashville, TN 37240 and IMAR, Bucharest, Romania }
\email{ionut.chifan@vanderbilt.edu}
\author{Jesse Peterson}
\address{Jesse Peterson, Vanderbilt University, 1326 Stevenson Center, Nashville, TN 37240}
\email{jesse.d.peterson@vanderbilt.edu} \subjclass{}

 \dedicatory{}
\date{\today}
\thanks{The first author's research is partially supported by NSF Grant 1001286}

\thanks{The second author's research is partially supported by NSF Grant 0901510 and a grant from the Alfred P. Sloan Foundation}

\begin{document}

\begin{abstract}

Using an approach emerging from the theory of closable derivations on von Neumann algebras, we exhibit a class of groups $\mathcal{CR}$
satisfying the following property: given \emph{any} groups $\G_1,\G_2\in \mathcal {CR}$, then \emph{any} free, ergodic, measure preserving
 action on a probability space $\G_1\times \G_2\car X$ gives rise to a von Neumann algebra with \emph{unique} group measure space
Cartan subalgebra. Pairing this result with Popa's Orbit Equivalence Superrigidity Theorem we obtain new examples of $W^*$-superrigid actions.
\end{abstract}

\maketitle \tableofcontents


\section*{Introduction and Notations}


The group measure space construction of Murray and von Neumann \cite{MvN1, MvN2} associates a finite von Neumann algebra, denoted by
$L^{\infty}(X)\rtimes \G$, to every measure preserving action $\G\car X$ of a countable group $\G$ on a standard probability space $X$. When
the action is free and ergodic, $L^{\infty}(X)\rtimes \G$ is a II$_1$ factor containing $L^{\infty}(X)$ as a \emph{Cartan subalgebra}, i.e., a
maximal abelian selfadjoint subalgebra with its normalizing group $\mathcal N_{L^{\infty}(X)\rtimes \G}(L^{\infty}(X))$ generating
$L^{\infty}(X)\rtimes \G$ as a von Neumann  algebra.

Two free, ergodic actions $\G\car X$ and $\La\car Y$ are called $W^*$-\emph{equivalent} if the corresponding
 group measure space von Neumann algebras are isomorphic. Also, two actions $\G\car X$ and $\La\car Y$ are said to be
\emph{conjugate} if there is a measure space isomorphism $\Phi:X\ra Y$ and a group isomorphism $\theta:\G\ra\La$ such that for all $\g\in \G$
we have $\Phi(\g x)=\theta(\g)\Phi(x)$ for almost every $x\in X$.

 Naturally, a conjugacy between two actions implements an isomorphism between the associated von Neumann algebras. Therefore the \emph{$W^*$-equivalence class}
  of an action always contains its \emph{conjugacy class}. Moreover, when $\G$ is infinite amenable, the \emph{$W^*$-equivalence class} of a free,
  ergodic action $\G\car X$  contains \emph{all} free, ergodic actions of \emph{all} infinite amenable groups \cite{connes76}, thus being (much)
  larger than its \emph{conjugacy class}. This is an instance when the von Neumann algebra arising from an action remembers little of the initial group/action data. The opposite phenomenon -
 when aspects of an action can be recovered from its von Neumann algebra - is labeled a \emph{$W^*$-rigidity} phenomenon.

An extreme form of rigidity for actions is \emph{$W^*$-superrigidity}. A free, ergodic, measure preserving action $\G\car X$ is called
$W^*$-\emph{superrigid} if its \emph{$W^*$-equivalence class} coincides with its \emph{conjugacy class}. In other words, whenever $\La\car Y$
is a free, ergodic, measure preserving action, any isomorphism between the von Neumann algebras $L^{\infty}(X)\rtimes\G$ and
$L^{\infty}(Y)\rtimes\La$ entails a conjugacy between the actions $\G\car X $ and $\La\car Y$. Producing examples of $W^*$-superrigid actions
$\G\car X$ is a difficult problem as it incorporates two rigidity phenomena which, even individually, are usually hard to establish.
\begin{enumerate}
\item \emph{Orbit Equivalence (OE) superrigidity}: If an arbitrary free, ergodic, p.m.p.\ action $\La\car Y$ is orbit equivalent with
$\G\car X$, i.e., there is an isomorphism $\Phi:X\ra Y$ such that $\Phi(\G x)=\La\Phi(x)$ for almost every $x\in X$, then the actions $\G\car
X$ and $\La\car Y$ are conjugated.
\item \emph{Uniqueness of group measure space Cartan subalgebras}: If the von Neumann algebra $N=L^{\infty}(X)\rtimes \G$  corresponding
to the action $\G\car X$ admits another group measure space decomposition $N=L^{\infty}(Y)\rtimes \La$, the group measure space Cartan
subalgebras $ L^{\infty}(X)$ and $ L^{\infty}(Y)$ are conjugated by a unitary in $N$.
\end{enumerate}

Due to a sustained effort over the last decade from both ergodic theory and Popa's \emph{deformations/rigidity theory} we have now a number of
examples of actions known to be OE-superrigid. See, for instance \cite{furman99a}, \cite{furman99b}, \cite{kida2006}, \cite{popa2006},
\cite{popa07b}, \cite{ioana08}, or \cite{kida-2009}.

However, the second problem was out of reach for an extended time. The breakthrough in this direction came only a few years ago with the
seminal work of Ozawa and Popa \cite{ozawapopa07}. Using techniques from deformation/rigidity theory, they showed that the von Neumann algebras
associated with profinite actions of products of nonamenable, free groups have unique Cartan subalgebras. Similar results covering more general
examples can be found in \cite{ozawapopa2010}.

Despite these important results, instances when an action simultaneously satisfies both forms of rigidity remained elusive. Recently, the
second author managed to prove the existence of actions satisfying the second type of rigidity while simultaneously virtually satisfying the
first type \cite{peterson2009}. More precisely, by developing an infinitesimal analysis for the resolvent deformations associated to closable
derivations on von Neumann algebras, it was shown that the von Neumann algebras arising from free, profinite actions of free products groups,
$\G_1\ast\G_2$ with $\G_1$ non-Haagerup, have unique group measure space Cartan subalgebra. Then, using a Baire category argument, results from
\cite{ioana08} and \cite{ozawapopa2010} were combined to show the existence of virtually $W^*$-superrigid actions.

Shortly after, Popa and Vaes proved that arbitrary free, ergodic actions of groups belonging to a large class of amalgamated free products give
rise to von Neumann algebras with unique group measure space Cartan subalgebra \cite{popavaes09}. When combining this with either Kida's
OE-superrigidity theorem in \cite{kida-2009} or with Popa's OE-superrigidity theorem in \cite{popa07b} it led to concrete examples of
$W^*$-superrigid actions, such an example being any free, mixing p.m.p.\ action of the group $PSL(n,\mathbb Z)\ast_{T_n} PSL(n,\mathbb Z)$. The
methods Popa and Vaes developed to prove their result brought a new insight to the deformation/rigidity technology through the introduction of
their ``transfer lemmas''. Outgrowths of these methods were used subsequently by Fima and Vaes to obtain examples of $W^*$-superrigid actions
covering other classes of groups, e.g., HNN-extensions \cite{fimavaes2010}.

Recently, Ioana was able to prove a striking result showing that the Bernoulli actions of \emph{any} property (T) group is $W^*$-superrigid
\cite{adi2010}.  While still heavily relying on deformation/rigidity theory, his methods brought a considerable amount of innovation at the
technical level. See also \cite{ioana-popa-vaes-2010} for further results in this direction.\vskip 0.05in

 Our paper focuses mainly on obtaining new
examples of actions which give rise to von Neumann algebras with unique group measure space Cartan subalgebras. To introduce the result we
consider the class $\mathcal{CR}$ of all countable, infinite conjugacy class groups $\G$ satisfying the following conditions:\begin{enumerate}
\item There exists an unbounded cocycle $c:\G \to \KK$ into a mixing representation;
\item There exists a non-amenable, infinite conjugacy class subgroup $\Omega$ of $\G$ such that the pair $(\G, \Omega)$ has relative property (T).
\end{enumerate}

Using an approach which derives from the theory of closable derivations on von Neumann algebras we show that \emph{any} free, ergodic p.m.p.\
action $\G\car X$ of \emph{any} group $\G$ belonging to $\mathcal {CR}$ gives rise to a von Neumann algebras with unique group measure space
Cartan subalgebra, thus adding to the examples found in \cite{popavaes09} and \cite{fimavaes2010}. The reader may also consult Theorem
\ref{uniquecartan3} for a similar result covering a class of groups larger than $\mathcal{CR}$.

 Moreover, when considering products of groups belonging to class $\mathcal{CR}$ we obtain the following:
\begin{thm*}[Corollary~\ref{cor:uniquegmscartan} below]\label{uniquegmscartan2}
 If $1 \le m \le 2$ let $\G=\G_1\times\cdots \times \G_m$ where $\G_i\in\mathcal{CR}$ for all $1 \le i \le 2$, and let $\G \car X$ be a free, ergodic p.m.p.\ action on a probability space $X$. If
there exists another free, p.m.p.\ action $\La\car Y$ on a probability space $Y$  such that $N=L^{\infty}(X)\rtimes\G = L^{\infty}(Y)\rtimes
\La$ then one can find a unitary $u\in N$ such that $uL^{\infty}(Y)u^* =L^{\infty}(X)$.
\end{thm*}

\noindent If $\G$ is a group with positive first $\ell^2$-Betti number ($\b^{(2)}_1(\G)>0$) there exists an unbounded cocycle into the left
regular representation \cite{bekkavalette, petersonthom2007}, and hence $\G$ satisfies condition (1) above. Therefore $\mathcal{CR}$ contains
all amalgamated free products $\G_1\ast_{\Omega}\G_2$ which satisfy the following properties (Proposition 3.1 in \cite{petersonthom2007}):
\begin{itemize}
  \item $\G_i$ are infinite groups such that $\b^{(2)}_1(\G_1)+ \b^{(2)}_1(\G_2)+\frac{1}{|\Omega|}>\b^{(2)}_1(\Omega)$;
  \item  $\G_1$ contains an infinite, i.c.c. group with property (T).
  \end{itemize}
Similarly, $\mathcal{CR}$ contains all HNN extensions  $HNN(\G,\Omega,\theta)$ which satisfy the following properties (Proposition 3.1 in
\cite{petersonthom2007}):
\begin{itemize}
  \item $\G$ is an infinite group such that $\b^{(2)}_1(\G)+\frac{1}{|\Omega|}>\b^{(2)}_1(\Omega)$;
  \item  $\G$ contains an infinite group with property (T).
  \end{itemize}
It also follows from Theorem 3.2 in \cite{petersonthom2007} that $\mathcal{CR}$ contains all groups $\G$ which have infinite subgroups $\G_1,
\G_2$ and a presentation $\G = \langle \G_1, \G_2 \ | \ r_1^{w_1}, \ldots, r_k^{w_k} \rangle$ for elements $r_1, \ldots, r_k \in \G_1 * \G_2$
and positive integers $w_1, \ldots, w_k$ satisfying the following properties:
\begin{itemize}
\item $r_i^l \not= e \in \G$, for $1 < l < w_i$;
\item $1 + \b_1^{(2)}(\G_1) + \b_1^{(2)}(\G_2) - \Sigma_{j = 1}^k \frac{1}{w_j} > 0$;
\item $\G_1$ contains an infinite, i.c.c. group with property (T).
\end{itemize}

The proof of the above theorem is obtained in several steps and it combines ideas from \cite{peterson2009} and some transfer lemmas \`a la
Popa-Vaes (see Lemma 3.2 in \cite{popavaes09}). We briefly explain below the idea behind the proof in the case $m=1$. First, a lemma similar to
Lemma 3.2 in \cite{popavaes09} is used to transfer, at the level of resolvent deformations, the rigidity part of the group $\G$ to ``large''
subsets of the mysterious group $\La$. In turn, this mild form of rigidity is used through a refinement of the infinitesimal analysis developed
in \cite{peterson2009} to show that the resolvent deformation converges to zero, uniformly on the unit ball of the ``mysterious'' Cartan
subalgebra $L^{\infty}(Y)$. Finally, this stronger ``rigid behavior'' of $L^{\infty}(Y)$ with respect to the the resolvent deformation is
exploited in the same way as in \cite{peterson2006} to completely locate the position of $L^{\infty}(Y)$ inside $N$.

While more technical, the proof of the product case follows the same general strategy. The difficulty however is that if $L^\infty(Y) \subset
L^\infty(X) \rtimes \G$ is a group-measure space Cartan subalgebra and $L^\infty(Y) \subset L^\infty(X) \rtimes \G_1$, then by \cite{jonespopa}
$L^\infty(Y)$ is also a Cartan subalgebra of $L^\infty(X) \rtimes \G_1$; however, there is no obvious reason for $L^\infty(Y)$ to again be a
\emph{group-measure space} Cartan subalgebra of $L^\infty(X) \rtimes \G_1$. This difficulty is overcome by developing a transfer property
(Lemma~\ref{transfer} below) which is applicable in the setting of products.

We do not know if the von Neumann algebras $N$ considered in the theorem above do have a unique, up to unitary conjugacy, Cartan subalgebra.

\vskip 0.05in

Using his deformation/rigidity theory, Popa discovered a natural class of OE-superrigid actions of product groups, showing in \cite{popa07b}
the following: Given any product of nonamenable groups $\G=\G_1\times\G_2$ and any countable $\G$-set $I$ such that for every $i\in I$, its
$\G_1$-orbit is infinite and its $\G_2$-stabilizer is amenable, the corresponding generalized Bernoulli action $\G\car(X,\mu)^I$, if free, is
OE-superrigid. The reader may notice that even though the actions considered are somewhat particular there is a large degree of generality at
the level of acting groups $\G_i$. Therefore, when letting the groups $\G_i$ to be in our class $\mathcal{CR}$, and combining it with
Theorem~\ref{uniquegmscartan2} above leads to the following examples of $W^*$-superrigid actions.

\begin{cor*}
Consider $\G=\G_1\times\G_2$
 where $\G_i\in\mathcal{CR}$ and let $I$ be a countable $\G$-set such that for all $i\in I$ the orbit $\G_1i$ is infinite and the stabilizer $\{\g\in \G_2 \ | \ \g i=i\}$
 is amenable. Then the corresponding generalized Bernoulli action $\G\car(X,\mu)^I$, if free, is $W^*$-superrigid.
\end{cor*}

Monod and Shalom considered in \cite{monodshalom} the class $\mathcal{C}_{\text{reg}}$ of all groups with nonvanishing second bounded
cohomology with coefficients in the left regular representation. In the same paper they proved, by using bounded cohomology methods, that any
free, irreducible, aperiodic, action of products of such groups is close to being OE-superrigid in the following sense: whenever this action is
orbit equivalent to any other free, mildly mixing action then the two actions must be conjugated. For a precise statement, as well as the
definitions of mild mixing and aperiodicity, the reader may consult Section~\ref{sec:applications} or Theorem 1.10 in \cite{monodshalom}. Monod
and Shalom proved the necessity of this condition for their statement; however it will be interesting to understand if their actions are
OE-superrigid when one assumes, in addition, that they are \emph{mixing}.

The class $\mathcal{C}_{\text{reg}}$ is quite large (see Section~\ref{sec:applications} below or Example 1.1 in \cite{monodshalom}) and it
intersects nontrivially with our class $\mathcal {CR}$. Basic examples of groups that belong to both classes include all free products of a
nontrivial group and an infinite property (T) group. For instance, if $|\G|\ge 2$ then we have $\G\ast SL(3,\mathbb Z )\in \mathcal{CR}\cap
\mathcal C_{\text{reg}} $.

Consequently, for such groups, Monod and Shalom's result together with our main theorem imply the following $W^*$-strong rigidity statement.

\begin{cor*}\label{strongrig2}
Let $\G=\G_1\times\G_2$ with $\G_i\in\mathcal{CR}\cap \mathcal C_{\text{reg}}$ for $i = 1,2$, and let $\G\car X$ be a free, irreducible,
aperiodic action. Suppose that $\La\car Y$ is mildly mixing action. If the action $\G\car X$ is $W^*$-equivalent with $\La\car Y$ then $\G\car
X$ and $\La\car Y$ are conjugate.
\end{cor*}

\noindent{\bf Organization of the paper.} This paper contains seven sections and one appendix. In the first section we review Popa's
intertwining techniques and we prove a few conjugacy results for actions of product groups. The first part of Section~\ref{sec:derivations}
collects important background on real, closable derivations along with their resolvent deformations. We show in Lemma~\ref{orthog} that
bimodules arising from mixing representations are mixing and, through an adaptation of technology from \cite{peterson2009}, we use this to
prove in Section \ref{sec:convergence} a criterion for the uniform convergence of the resolvent deformations on certain subalgebras (Theorem
~\ref{derivunifconv}). In a similar fashion with Lemma 3.2 in \cite{popavaes09} we prove in Section~\ref{sec:transfer} a transfer
Lemma~\ref{transfer} for actions of product of groups in class $\mathcal{CR}$.  In  turn, this transfer lemma is used in combination with
Theorem ~\ref{derivunifconv} to prove the unique group measure Cartan decomposition result in Corollary~\ref{cor:uniquegmscartan}. In Section
\ref{sec:applications} we use Corollary~\ref{cor:uniquegmscartan} to derive our main applications to $W^*$-superrigidity,
Corollaries~\ref{w-superrigidity} and~\ref{strongrig}. In the last section we exhibit more examples of von Neumann algebras with unique group
measure space Cartan subalgebras (Theorem \ref{uniquecartan3}). \vskip 0.1in

\noindent {\bf Notations.} In this paper all finite von Neumann algebras $N$ that we will be working with are assumed to be endowed with a
normal faithful tracial state, which we will denote by $\tau$. The trace $\tau$ induces a norm on $N$ by letting
$\|x\|_2=\tau(x^*x)^{\frac{1}{2}}$. As usual, $L^2N$ denotes the $\|\cdot \|_2$-completion of $N$. A Hilbert space $\mathcal H$ is a $N$-{\it
bimodule} if it is equipped with commuting left and right Hilbert $N$-module structures.

Given a von Neumann subalgebra $Q\subset N$ we denote by $E_{Q}:N\ra N$ the unique $\tau$-preserving {\it conditional expectation} onto $Q$. If
$e_Q$ is the orthogonal projection of $L^2N$ onto $L^2(Q)$ then $\langle N,e_Q\rangle $ denotes the \emph{basic construction}, i.e., the von
Neumann algebra generated by $N$ and $e_Q$ in $\mathcal B(L^2N)$. The span of $\{xe_Qy \ | \ x,y\in N\}$ forms a dense $*$-subalgebra of
$\langle N,e_Q\rangle$ and there exists a semifinite trace $Tr:\langle N,e_Q\rangle\ra \mathbb C$ given by the formula $Tr(xe_Qy)=\tau(xy)$ for
all $x,y\in N$. We denote by $L^2\langle N,E_Q\rangle$ the Hilbert space obtained with respect to this trace.

The {\it normalizer of $Q$ inside $N$}, denoted $\mathcal N_{N}(Q)$, consists of all unitary elements $u\in \mathcal U(N)$ satisfying
$uQu^*=Q$. A maximal abelian selfadjoint subalgebra $A$ of $N$, abbreviated MASA, is called  a \emph{Cartan subalgebra} if the von Neumann
algebra generated by its normalizer in $N$, $\mathcal N_{N}(A)''$ is equal to $N$. Also, $A$ is called \emph{semiregular} if $\mathcal
N_{N}(A)''$ is a subfactor of $N$.

If $\G\car^{\sigma}A $ is a trace preserving action by automorphisms of a countable group $\G$ on a finite von Neumann algebra $A$
 we denote by $N=A\rtimes_{\sigma}\G$ the crossed product von Neumann algebra associated with the action. When no confusion will arise we will
 drop the symbol $\sigma$.  Given a subset $F\subset \G$, we will denote by $P_F$ the
 orthogonal projection on the closure of the span of $\{au_{\g} \ | \ a\in A;\g\in F \}$.

Throughout this paper $\omega$ denotes a free ultrafilter on $\mathbb N$. Also given $(N,\tau)$ a finite von Neumann algebra we denote by
$(N^{\omega},\tau^{\omega})$ its ultrapower algebra, i.e., $N^{\omega}=\ell^{\infty}(\mathbb N ,N)/\mathcal I$ where the trace is defined as
$\tau^{\omega}((x_n)_n)=\lim_{n\ra\omega}\tau(x_n)$ and $\mathcal I$ is the ideal consisting of all $x\in \ell^{\infty}(\mathbb N, N )$ such
that $\tau^{\omega}(x^*x)=0$. Notice that $N$ embeds naturally into $N^{\omega}$ by considering constant sequences. Many times when working
with $N=A\rtimes \G$ we will consider the subalgebra $A^{\omega}\rtimes \G$ of $N^{\omega}$. For reader's convenience we remark that every
element $x=(x_n)_n\in A^{\omega}\rtimes \G$ satisfies that $\inf_{F\subset \G,\text{finite}}\lim_{n\ra\omega}\|x_n-P_F(x_n)\|_2=0$.

\vskip 0.1in

\noindent{\bf Acknowledgement.} We are grateful to Stefaan Vaes for numerous useful comments and for pointing out an error in an early version
of this paper.

\section{Intertwining techniques}\label{sec:intertwining}

We start this section by reviewing Popa's intertwining techniques from \cite{popa2004}. Given $N$ a finite von Neumann algebra, let $P\subset
fNf$, $Q\subset N$ be diffuse subalgebras for some projection $f\in N$. One say that \emph{a corner of $P$ can be embedded into $Q$ inside $N$}
if there exist two nonzero projections $p\in P$, $q\in Q$, a nonzero partial isometry $v\in pNq$, and a $*$-homomorphism $\psi:pPp\ra qQq$ such
that $v\psi(x)=xv$ for all $x\in pPp$. Throughout this paper we denote by $P\prec_{N}Q$ whenever this property holds and by $P\nprec_{N}Q$
otherwise.

Popa established efficient criteria for the existence of such intertwiners (Theorems 2.1-2.3 in \cite{popa2004}). Particularly useful in
applications is the analytic criterion described in Corollary 2.3 of \cite{popa2004}. Considering the case of crossed-products von Neumann
algebras, the next proposition is a reformulation of Popa's result using the ultrapower algebras setting.

\begin{prop}\label{intertwining}Let $\G$ be a countable group, $A$ a finite von Neumann algebra and $\G\car A$ is a trace preserving action.
Suppose that $N=A\rtimes\G$ and $B\subset N$ is a subalgebra which is either abelian or a II$_1$ factor. If $q\in B'\cap N $ is a nonzero
projection then the following are equivalent:
\begin{enumerate}
\item $Bq\nprec_{N}A$.
\item For any nonzero projection $p\in (B'\cap N )^{\omega}$  with $p\le q$ we have $B^{\omega}p\nsubseteq A^{\omega}\rtimes \G $.
\item For any nonzero projection $p\in B'\cap N $ with $p\le q$ we have $B^{\omega}p\nsubseteq A^{\omega}\rtimes \G $.
\end{enumerate}
Moreover, if $\mathcal N_{N}(B)'\cap  N=\mathbb C 1  $ (for instance when $B$ is a semiregular MASA) then the following are equivalent:

$(1)'\quad  B\prec_{N}A$.

$(2)'\quad B^{\omega}\subseteq A^{\omega}\rtimes \G $.

\end{prop}

\begin{proof}First we prove $(1)\Rightarrow (2)$. Assuming $Bq\nprec_{N}A$, by Popa's intertwining result \cite{popa2004} there exists a sequence of
unitaries $u_n\in\mathcal U (B)$ such that for all $x,y\in N$ we have $\|E_A(xu_nqy)\|_2\ra 0 $ as $n \ra \infty$. This easily implies
$E_{A^{\omega}\rtimes \G}(uq)=0$, where $u=(u_n)_n\in N^{\omega}$.  If $p\in (B'\cap N )^{\omega}\cap (A^{\omega}\rtimes \G)$ is a projection
such that $p\le q$, then $E_{A^{\omega}\rtimes \G}(up)=E_{A^{\omega}\rtimes \G}(uqp)=E_{A^{\omega}\rtimes \G}(uq)p=0$ and hence
$B^{\omega}p\nsubseteq A^{\omega}\rtimes \G$ unless $p=0$.

The implication $(2)\Rightarrow (3)$ is obvious and therefore to finish the proof it only remains to show $(3)\Rightarrow(1)$.  We will prove
this by contraposition. Assuming $Bq\prec_{N}A$ one can find nonzero projections $rq\in Bq$, $p\in A$, an injective $*$-homomorphism $\psi:
rBrq\ra pAp$ and nonzero partial isometry $v\in rqN$ such that $v\psi(x)=xv$ for all $x\in rBrq$. The last equation implies that $vv^*\in
(rBrq)'\cap rqNrq $ and therefore we have the following containment:
\begin{equation}\label{100}rBrvv^*=v\psi(rBr)v^*\subseteq vAv^*.\end{equation}

We notice that there exists nonzero projection $q'\in B'\cap N$ with $q'\le q$ such that $vv^*= rq'$ and combining this with relation
(\ref{100}) we obtain that
\begin{equation}\label{101}(rBr)^{\omega}q'\subseteq A^{\omega}\rtimes \G.\end{equation}

If $B$ is a II$_1$ factor then by passing to a subprojection we may assume that $\tau_N(r)=\frac{1}{k}$ for some positive integer $k$. Also for
every $i,j\in \{1,...,k\}$ there exist partial isometries $e_{ij}\in B $ such that $e_{11}=r$, $e^*_{ij}=e_{ji}$, $e_{ij}e_{ji}=e_{ii}\in
\mathcal P(B)$ and $\sum_i e_{ii}=1$. If $(x_n)_n\in B^{\omega}$ then using the above relations in combination with $q'\in B'\cap N$ we
have that \begin{eqnarray}\nonumber(x_n)_n(q')_n&=&(x_nq')_n=(\sum_{i,j}e_{ii}x_ne_{jj}q')_n=\sum_{i,j}(e_{i1}e_{1i}x_ne_{j1}e_{1j}q')_n\\
\label{102}&=&\sum_{i,j}(e_{i1})_n(e_{1i}x_ne_{j1})_n(q')_n(e_{1j})_n.\end{eqnarray}

One can easily see that $(e_{1i}x_ne_{j1})_n\in (rBr)^{\omega}$ and combining this with relations (\ref{101}) and (\ref{102}) we conclude that
$(x_n)_n(q')_n\in A^{\omega}\rtimes \G$, thus showing that $B^{\omega}q'\subseteq A^{\omega}\rtimes \G$. Therefore, in both cases ($B$ abelian
and $B$ a II$_1$ factor) we can assume that there exists a nonzero projection $p\in B'\cap N$ with $p\le q$ such that $B^{\omega}p\subseteq
A^{\omega}\rtimes \G$, finishing the proof of implication $(3)\Rightarrow(1)$.

If $B\prec_{N}A$ then by equivalence $(1)\Leftrightarrow (3)$ one can find a nonzero projection $ p\in B'\cap N$ such that
$B^{\omega}p\subseteq A^{\omega}\rtimes \G $. Conjugating by $u\in\mathcal N_M(B)\subset \mathcal N_M(B'\cap N)$ we obtain
$B^{\omega}upu^*\subseteq A^{\omega}\rtimes \G $, for all $u\in\mathcal N_M(B)$, hence $B^{\omega}p_0\subseteq A^{\omega}\rtimes \G $ where
$p_0=\vee_{u\in\mathcal N_M(B)}upu^*\in B'\cap N$. One easily sees that $p_0$ commutes with $\mathcal N_M(B)$ and therefore it belongs to
$\mathcal N_M(B)'\cap N$. By assumption we have $\mathcal N_M(B)'\cap N=\mathbb C 1$, we then conclude that $p_0=1$ and therefore
$B^{\omega}\subseteq A^{\omega}\rtimes_{\sigma} \G $. This shows $(1)'\Rightarrow (2)'$ and the reversed implication follows immediately from
the equivalence $(1)\Leftrightarrow (3)$.
\end{proof}

Before stating the next intertwining result we introduce some notation. Given a product group $\G=\G_1\times \G_2\times \cdots \times \G_m$,
for every $1\le i\le m$, we denote by $\G_{(i)}$ the subgroup of $\G$ consisting of all elements in $\G$ whose $i^{\text{th}}$ coordinate is
trivial, so that we have the natural identification $\G=\G_i\times \G_{(i)}$ for each $1\le i\le m$.

\begin{cor}\label{reducingintertwininginproducts}Let $\G=\G_1\times \G_2\times\cdots \times \G_m$ and let $\G\car A$ be a trace preserving action
on a finite von Neumann algebra $A$. Assume that $B\subseteq A\rtimes \G=N$ is a subalgebra satisfying one of the following conditins:
\begin{enumerate}\item $B$ is a II$_1$ factor such that $Bq\prec_N A\rtimes \G_{(i)}$ for all $q\in B'\cap N$ and $1\le i\le m$;
\item $B$ is a semiregular MASA such that $B\prec_N A\rtimes \G_{(i)}$ for all $1\le i\le m$.\end{enumerate}

Then $B\prec_N A$.
\end{cor}

\begin{proof} Notice that for every $1\le i\le m$ the algebra $N$ can be seen as $(A\rtimes \G_{(i)})\rtimes \G_i$ with $\G_i$ acting trivially on
$\G_{(i)}$. First we assume situation (1), i.e. $B$ is a II$_1$ factor such that for all $q\in B'\cap N$ we have $Bq\prec_N A\rtimes \G_{(i)}$.
Therefore, by Proposition \ref{intertwining} there exists nonzero projection $q_i\in B'\cap N $ such that
\begin{eqnarray}\label{502}B^{\omega}q_i\subset (A\rtimes \G_{(i)})^{\omega}\rtimes
\G_i.\end{eqnarray}  Then we let $q_i\in B'\cap N $ to be maximal such that $B^{\omega}q_i\subset (A\rtimes \G_{(i)})^{\omega}\rtimes \G_i$ and
below we argue that $q_i=1$. Assuming the contrary, we have $0\neq 1-q_i\in B'\cap N$ and since by initial assumption
$B(1-q_i)\prec_{N}A\rtimes \G_{(i)}$ there exists a nonzero projection $p_i\in B'\cap N $ with $p_i\le 1-q_i$ such that
\begin{equation*}B^{\omega}p_i\subset (A\rtimes \G_{(i)})^{\omega}\rtimes \G_i.\end{equation*} Combining this with (\ref{502}) we get that
$B^{\omega}(q_i+p_i)\subset (A\rtimes \G_{(i)})^{\omega}\rtimes \G_i$ which obviously contradicts the maximality of $q_i$.

Altogether we obtained that for all $1\le i\le m $ we have $B^{\omega}\subset (A\rtimes \G_{(i)})^{\omega}\rtimes \G_i$ and hence we have
\begin{eqnarray*}B^{\omega}\subset\bigcap^m_{i=1}(A\rtimes \G_{(i)})^{\omega}\rtimes \G_i.\end{eqnarray*}

If we let $x=(x_n)_n\in\bigcap^m_{i=1}(A\rtimes \G_{(i)})^{\omega}\rtimes \G_i$ then for every $1\le i \le m$ we have that
 \begin{equation*}\inf_{F_i\subset \G_i \text{, finite }}\lim_{n\ra\omega}\|P_{F_i}(x_n)-x_n\|_2=0.\end{equation*}

Using these relations in combination with triangle inequality we obtain that \begin{eqnarray*}\inf_{F_1\times\cdots\times F_m\subset \G\text{,
finite }}\lim_{n\ra\omega}\|P_{F_1}\circ\cdots \circ P_{F_m}(x_n)-x_n\|_2=0,\end{eqnarray*} and since $P_{F_1}\circ\cdots \circ
P_{F_2}(x_n)=P_{F_1\times\cdots \times F_m}(x_n)$ we conclude that $x\in A^{\omega}\rtimes (\G_1\times\cdots \times\G_m)$.

Hence $B^{\omega}\subset\bigcap^m_{i=1}(A\rtimes \G_{(i)})^{\omega}\rtimes \G_i=A^{\omega}\rtimes \G=A^{\omega}\rtimes \G$ and by
Proposition~\ref{intertwining} we have that $B\prec_N A$.

For case (2) just notice that by the second part of Proposition \ref{intertwining} we have that $B^{\omega}\subset (A\rtimes
\G_{(i)})^{\omega}\rtimes \G_i$ for all $1\le i\le m$ and therefore the conclusion follows from above.
\end{proof}

We end this section by recalling Popa's conjugacy criterion for Cartan subalgebras which will be used in the sequel.

\begin{thm}[Appendix 1 in \cite{popa2001}]\label{intertwining-conj}
Let $N$ be a II$_1$ factor and $A,B\subset N$ two semiregular MASAs. If $B_0 \subset B$ is a von Neumann subalgebra such that $B_0' \cap N =
B$, and $B_0\prec_NA$, then there exists a unitary $u\in N$ such that $uAu^*=B$.
\end{thm}


\section{Background on derivations}\label{sec:derivations}

Let $\G$ be a countable group, and assume that $\G\car^{\sigma}A $ is a trace preserving action on a finite von Neumann algebra $A$. Given an
orthogonal representation $\pi:\G \ra \mathcal{O}(\HH)$, it was shown in \cite{sauvageot} that each $1$-cocycle associated with $\pi$ gives
rise naturally to a closable, real derivation on $N=A\rtimes\G$. This means there is a linear map $\dd:D(\dd)\ra \mathcal H_{\pi}$ where
$D(\dd)$ is a weakly dense $*$-subalgebra of $N$ and $ \mathcal H_{\pi}$ is a Hilbert $N$-bimodule satisfying the following properties:
\begin{itemize}
\item $\dd(xy)=x\dd(y)+\dd(x)y$ for all $x,y\in D(\dd)$;
\item $\dd$ is closable as an unbounded operator from $L^2N$ to $\mathcal H_{\pi}$;
\item There exists $J: \mathcal H_{\pi}\ra \mathcal H_{\pi}$ antilinear involution such that $J(x\xi y)=y^* \xi x^*$ and $J(\dd(x))=\dd(x^*)$ for all
$x,y,z\in D(\dd), \xi\in \mathcal H_{\pi}.$
\end{itemize}
We briefly recall this construction below.

The Hilbert $N$-bimodule $\HH_{\pi}$ is defined as $\HH_{\pi}=\HH\overline{\otimes} L^2N$ where the left and right actions of $N$ on
$\HH_{\pi}$ satisfy
\begin{equation}\label{72}(au_{\g})\cdot (\xi \otimes\eta)\cdot (bu_{\l})=(\pi(\g)\xi)\otimes ((au_{\g})\eta (bu_{\l})),\end{equation} for all
$a,b\in A$, $\xi\in\mathcal H_{\pi}$, $\eta\in L^2N$ and $\g,\l\in \G$.

Given $c:\G\ra\HH$ an additive 1-cocycle for $\pi$, i.e., $c(\g\l)= c(\g)+\pi(\g)c(\l)$ for all $\g,\l\in \G$, we define
$\dd:A\rtimes_{alg}\G\ra\HH_{\pi}$ by linearly extending formula $\dd(au_{\g})=c(\g)\otimes (au_{\g})$, where $a\in A$, $\g\in\G$. It is
straight forward to verify that this map is a closable, real derivation on $N$.

\vskip 0.1in

Consider the Hilbert space $\tilde{\HH}_{\pi}=\HH_{\pi}\overline{\otimes} L^2N$ and observe that this is an $N\otimes N$-bimodule with respect
to the left and right actions which satisfy  \begin{equation}(x\otimes z)(\mu\otimes \eta )(y\otimes t)=(x\cdot \mu\cdot  y)\otimes (z\eta
t),\end{equation} $\text{ for all }x,y,z,t \in N, \mu\in \HH_{\pi}\text{ and } \eta\in L^2N.$

We define a linear map $\tilde{\dd}:(A\rtimes_{alg} \G)\otimes N\ra\tilde{\HH}_{\pi}$ by linearly extending the formula
\begin{equation}\tilde{\dd}((au_{\l})\otimes x)=\dd(au_{\l})\otimes x,\end{equation} where $a\in A,\g\in \G$ and $x\in N$. Since $\dd$ is a
closable, real derivation on $N$ we have that $\tilde{\dd}$ is a closable, real derivation on $N\otimes N$. In fact $\tilde{\dd}$ is nothing
but the tensor product derivation $\dd \otimes 0$ as defined in Section 4.2 of \cite{petersonsinclair09}.

Associated to each closable, real derivation $\dd$ is the resolvent deformation given by
\begin{equation}\rho_{\alpha}=\frac{\alpha}{\alpha+\dd^*\overline{\dd}},\quad\zeta_{\alpha}=(\rho_{\alpha})^{\frac{1}{2}},\text{ for all } \alpha>0.\end{equation}
From \cite{sauvageot,peterson2006} it follows that $\rho_{\alpha}$ and $\zeta_{\alpha}$ are two families of $\tau$-symmetric, unital,
completely positive maps on $N$ such that for all $x\in N$ we have that $\|x-\rho_{\alpha}(x)\|_2\ra 0$ and $\|x-\zeta_{\alpha}(x)\|_2\ra 0$ as
$\alpha\ra \infty$.

We let $(\mathcal H_{\alpha},\xi_{\alpha})$ be the pointed $N$-bimodule corresponding to the map $\zeta_{\alpha}$ (see, for example,
\cite{popacorr86})and define the map $\dd_{\alpha}:N\ra \mathcal H_{\alpha}\overline{\otimes}_N{\HH}_{\pi}\overline{\otimes}_N\mathcal
H_{\alpha}$ by the formula
\begin{equation}\dd_{\alpha}(x)=\alpha^{-\frac{1}{2}}\xi_{\alpha}\otimes_N(\dd\circ\zeta_{\alpha})(x)\otimes_N \xi_{\alpha},\end{equation} where $\otimes_N$
denotes Connes' fusion product of $N$-bimodules. After a closer examination the reader may observe that when $\dd$ is comes from a cocycle $c$
as described above then the $N$-bimodule $\mathcal H_{\alpha}$ is nothing but $\mathcal H_{\pi^c_{\alpha}}$, where $\pi^c_{\alpha}$ is the the
representation of of $\G$ which corresponds to the positive definite function $\g\ra \sqrt{\frac{\alpha}{\alpha+\|c(\g)\|^2}}$.

Likewise, associated to  $\tilde{\dd}$ are two families of $\tau$-symmetric, unital, completely positive maps on $N\otimes N$ given by
\begin{equation}\tilde{\rho}_{\alpha}=\frac{\alpha}{\alpha+\tilde{\dd}^*\overline{\tilde{\dd}}}=\rho_{\alpha}\otimes
\id,\quad\tilde{\zeta}_{\alpha}=(\tilde{\rho}_{\alpha})^{\frac{1}{2}}=\zeta_{\alpha}\otimes\id,\text{ for all }\alpha>0.\end{equation}  Define
the Hilbert space $\tilde{\HH}^\alpha_{\pi}=(\mathcal H_{\alpha}\overline{\otimes}_N{\HH}_{\pi}\overline{\otimes}_N\mathcal
H_{\alpha})\overline{\otimes} L^2N$ which we endow with the natural $N\otimes N$-bimodule structure and consider $\tilde{\dd}_{\alpha}:N\otimes
N\ra \tilde{\HH}^\alpha_{\pi}$ the map given by the formula $\tilde{\dd}_{\alpha}=\dd_{\alpha}\otimes \id$.

In the next two propositions we summarize a few basic properties of ${\dd}_{\alpha}$ that will be used extensively throughout this paper. For
proofs of these facts the reader may consult Section 2 in \cite{peterson2006} or Section 4 in \cite{ozawapopa2010}.
\begin{prop} Using the above notation suppose $x\in N$.  Then we have the following:
\begin{eqnarray}
\label{71}&&\|x-\rho_{\alpha}(x)\|_2\le\|\dd_{\alpha}(x)\|\le\|x-\rho_{\alpha}(x)\|_2^{\frac{1}{2}};\\
&&\dd_{\alpha}\text{ is a contraction, i.e.,  } \|\dd_{\alpha}(x)\|\le\|x\|_2\le\|x\|_{\infty};\\
 &&\text{ The function }\alpha\mapsto\|\dd_{\alpha}(x)\|^2=\tau((\id-\rho_{\alpha})(x)x^*)\text{ is decreasing}.
\end{eqnarray}
\end{prop}

\begin{prop}\label{70} Using the above notation, for all $\alpha>0$ and $a,x\in N$
we have the following inequalities:
\begin{eqnarray*}
\|a\dd_{\alpha}(x)-\dd_{\alpha}(ax)\|\le 50\|x\|_{\infty}\|a\|^{\frac{1}{2}}_{\infty}\|\dd_{\alpha}(a)\|^{\frac{1}{2}};\\
\|\dd_{\alpha}(xa)-\dd_{\alpha}(x)a\|\le 50\|x\|_{\infty}\|a\|^{\frac{1}{2}}_{\infty}\|\dd_{\alpha}(a)\|^{\frac{1}{2}}.
\end{eqnarray*}

\end{prop}

As noted before, $\rho_{\alpha}$ converges pointwise to the identity on $N$ with respect to $\|\cdot\|_2$ and therefore, by (\ref{71}) above,
this is equivalent to $\dd_{\alpha}$ converging pointwise to zero on $N$. The next lemma shows that in fact this pointwise convergence holds
even when passing to certain (larger) ultrapower algebras associated with $N$.
\begin{lem}\label{convergence} Let $\Sigma$ be normal subgroup of $\G$ and assume that $\G$ admits an unbounded $1$-cocycle which vanishes
on $\Sigma$. Let $\dd$ be the closable real derivation associated to this cocycle, as described above. For every $x=(x_n)_n\in (A\rtimes
\Sigma)^{\omega}\vee N$ we have
\begin{equation}\label{50}\lim_{\alpha\ra\infty}\lim_{n\ra\omega}\|\dd_{\alpha}(x_n)\|=0.\end{equation}
\end{lem}
\begin{proof} Fix an arbitrary $\e>0$. Since $x=(x_n)_n\in (A\rtimes \Sigma)^{\omega}\vee N$  and $\Sigma$ is normal in $\G$ we
can find a finite set $F\subset \G$ of cosets representatives of $\Sigma$ in $\G$ such that
\begin{equation}\label{51}\lim_{n\ra\omega}\|x_n-R_F(x_n)\|_2\le\frac{\e}{2},\end{equation}
where $R_F$ denotes the projection from $L^2N$ onto the $L^2$-closure of $sp\{au_{\g} \ | \ a\in {A\rtimes \Sigma},\g\in F\}$.

Also, since $\dd_{\alpha}$ converges pointwise to zero and $F$ is finite, there exists $\alpha_{\e}$ such that for all $\g\in F$ and
$\alpha>\alpha_{\e}$ we have \begin{equation}\|\dd_{\alpha}(u_{\g})\|\le\frac{\e}{2|F|\|x\|_{\infty}}.\end{equation}

Since the cocyle vanishes on $\Sigma$ it follows that $\dd_{|A\rtimes \Sigma}=0$ and using this in combination with inequalities
$\|\dd_{\alpha}(m)\|\le\|m\|_2$ and $\|E_{A\rtimes \Sigma}(x_nu^*_{\l})\|_{\infty}\le \|x_n\|_{\infty}$, for all $n\in \mathbb N$ and
$\alpha>0$ we have:
\begin{eqnarray*}\|\dd_{\alpha}(x_n)\|&\le & \|x_n-R_F(x_n)\|_2+\|\dd_{\alpha}(\sum_{\g\in F}E_{A\rtimes \Sigma}(x_nu^*_{\l})u_{\l})\|\\
&\le&  \|x_n-R_F(x_n)\|_2+\sum_{\g\in F}\|E_{A\rtimes \Sigma}(x_nu^*_{\l})\dd_{\alpha}(u_{\l})\|\\
&\le & \|x_n-R_F(x_n)\|_2+\sum_{\g\in F}\|x\|_{\infty}\|\dd_{\alpha}(u_{\l})\|.
\end{eqnarray*}
Taking $\lim_{n\ra\omega}$ above and combining this with (\ref{50}) and (\ref{51}) we obtain that $\lim_{n\ra\omega}\|\dd_{\alpha}(x_n)\|\le
\e$ for all $\alpha>\alpha_{\e}$. Since $\e>0$ was arbitrary we obtain the desired equality.
\end{proof}
In the previous lemma the normality assumption on $\Sigma$ was made only for convenience. The same statement holds if we drop it.

 The $N$-bimodules $\mathcal H_{\pi}$ coming
from representations $\pi$ often inherits many useful properties from $\pi$. For instance, as observed in
\cite{petersonsinclair09,peterson2009}, if $\pi$ is a mixing representation, then $\mathcal H_{\pi}$ is mixing relative to $A$. More generally,
this also holds in the setting of groups that admit mixing representations with respect to subgroups. Below, for reader's convenience, we
include a proof of this fact. Given a group $\G$ with a subgroup $\Sigma < \G$, and a representation $\pi: \G \to \UU(\KK)$, we say that $\pi$
is mixing relative to $\Sigma$ if $\langle \pi(\g_n) \xi, \eta \rangle \to 0$ whenever the sequence $\g_n$ escapes any left-right coset of $\Sigma$ in $\G$. 

\begin{lem} \label{14} Let $\Sigma< \G$ be groups and let $\pi:\G\ra \UU(\KK)$ be a representation which is mixing relative to $\Sigma$.  If  $\xi, \eta \in \mathcal K_\pi=
\mathcal K\overline{\otimes} L^2N$ and $(c_n)_n,(d_n)_n\in N^{\omega}$ such that $(c_n)_n \perp N({A \rtimes \Sigma})^{\omega}N$ in
$L^2(N^{\omega})$, then
\begin{equation}\label{17}\lim_{n\ra\omega}\langle c_n\xi d_n,\eta\rangle =0.\end{equation}

\end{lem}
 \begin{proof} Using basic approximations in $\mathcal K_{\pi}$,  it
 suffices to prove (\ref{17}) only for elements of the form $\xi=\xi_1\otimes \xi_2$ and $\eta=\eta_1\otimes \eta_2$ with $\xi_1,\eta_1\in\mathcal K$ and
$\xi_2,\eta_2\in N$. As vectors of this form are left-bounded (see Chapter 1 in \cite{popacorr86}) we may use the Fourier expansion
$c_n=\sum_{\g\in \G}c^n_{\g}u_{\g}$ and the $N$-bimodule structure of $\mathcal K$ to show that
\begin{eqnarray}\label{20}
\langle c_n\xi d_n,\eta\rangle &=&\langle \sum_{\g\in \G} (\pi(\g)\xi_1)\otimes ((c^n_{\g}u_{\g})\xi_2 d_n),\eta_1\otimes \eta_2\rangle\\
\nonumber &=&  \sum_{\g\in \G} \langle\pi(\g)\xi_1,\eta_1\rangle\langle(c^n_{\g}u_{\g})\xi_2 d_n,\eta_2\rangle\\
\nonumber &=&  \tau((\sum_{\g\in \G} \langle\pi(\g)\xi_1,\eta_1\rangle c^n_{\g}u_{\g})\xi_2 d_n\eta^*_2),
\end{eqnarray}
where the element $\sum_{\g\in \G} \langle\pi(\g)\xi_1,\eta_1\rangle c^n_{\g}u_{\g}$ belongs to $N$ by \cite{decanniere-haagerup}.

Fix an arbitrary $\e>0$. Since the representation $\pi$ is mixing relative to $\Sigma$ there exists a finite set $F\subset \G$ such that
$|\langle\pi(\g)\xi_1,\eta_1\rangle|<\e$ for
every $\g\in \G\setminus F\Sigma F$. Therefore, using this in conjunction with (\ref{20}) and the Cauchy-Schwarz inequality we obtain that
\begin{eqnarray}\nonumber
 |\langle c_n\xi d_n,\eta\rangle| &\le& |\tau((\sum_{\g\in F\Sigma F}\langle\pi(\g)\xi_1,\eta_1\rangle c^n_{\g}u_{\l})\xi_2 d_n\eta^*_2)|+|\tau((\sum_{\g\in \G\setminus F\Sigma F}\langle\pi(\g)\xi_1,\eta_1\rangle c^n_{\g}u_{\l})\xi_2 d_n\eta^*_2)|\\
\nonumber &\le& |\tau((\sum_{\g\in F\Sigma F}\langle\pi(\g)\xi_1,\eta_1\rangle c^n_{\g}u_{\l})\xi_2 d_n\eta^*_2)|+\|\sum_{\g\in \G\setminus F\Sigma F}\langle\pi(\g)\xi_1,\eta_1\rangle c^n_{\g}u_{\l}\|_2\|\xi_2 d_n\eta^*_2\|_2\\
\nonumber &\le& \|\xi_1\|\|\eta_1\|\|\xi_2d_n\eta^*_2\|_2(\sum_{\g\in F\Sigma F}\|c^n_{\g}\|^2_2)^{\frac{1}{2}}+\e\|c_n\|_2\|\xi_2 d_n\eta^*_2\|_2 \\
\nonumber &\le& \|d_n\|_2\|\xi_1\|\|\eta_1\|\|\xi_2\|_{\infty} \|\eta^*_2\|_{\infty}(\sum_{\g\in F\Sigma F}\|c^n_{\g}\|^2_2)^{\frac{1}{2}}+\e
\|c_n\|_2\|d_n\|_2\|\xi_2\|_{\infty} \|\eta^*_2\|_{\infty}.
\end{eqnarray}

Since $(c_n)_n \perp N({A \rtimes \Sigma}^{\omega}) N$ and $F$ is finite we have that $\lim_{n\ra\omega}\sum_{\g\in F\Sigma F
}\|c^n_{\g}\|^2_2=0$. Combining this with the above inequality we conclude that
\begin{equation*}\lim_{n\ra\omega}|\langle c_n\xi d_n,\eta\rangle| \le\e \lim_{n\ra\omega}\|c_n\|_2\|d_n\|_2\|\xi_2\|_{\infty} \|\eta^*_2\|_{\infty}\le\e C \|\xi_2\|_{\infty} \|\eta^*_2\|_{\infty}. \end{equation*}
Since $\e>0$ was arbitrary, we have $\lim_{n\ra\omega}|\langle c_n\xi d_n,\eta\rangle| =0$.\end{proof}

\vskip 0.1in For further use, we recall from \cite{peterson2006} the following convergence property for the resolvent deformations. We also
include a proof for the sake of completeness.

\begin{thm}[compare with Theorem 4.5 in \cite{peterson2006}]\label{unifconvonnormalizer} Let $N$
be as above and let $B\subset N$ a subalgebra and $p\in B'\cap N$ such that $Bp\nprec_N A\rtimes \Sigma$. If $\dd_\alpha$ converges uniformly
to zero on $ (B)_1$ then $\dd_\alpha$ converges uniformly to zero on $p(\mathcal N_N(B)'')_1$.
\end{thm}
\begin{proof}

\noindent Let $u\in \mathcal N_N(B)$. Since by assumption $\dd_\alpha$ converges uniformly to zero on $ (B)_1$ and $u\in \mathcal N_N(B)$ then
for every $s\in\mathbb N $ there exists $\beta^1_s>0$ such that for all $\alpha >\beta^1_s$ and all $n\in \mathbb N$ we have
\begin{eqnarray}\label{723}\|\dd_\alpha(v_n)\|\le \frac{1}{s}\text{ and } \|\dd_\alpha(u^*v^*_nu)\|\le \frac{1}{s}.\end{eqnarray}

\noindent Also since $\dd_\alpha$ converges pointwise to zero there exists $\beta^2_s>0$ such that for all $\alpha >\beta^2_s$ we have
\begin{eqnarray}\label{724}\|\dd_\alpha(p)\|\le \frac{1}{s}.\end{eqnarray}

\noindent Notice that since $Bp\nprec_N A$ by Corollary 2.3 in \cite{popa-2004} there exists a sequence of unitaries $v_n\in B$ such that for
all $x,y\in N$ we have $\|E_A(xv_npy)\|_2\ra 0$ as $n\ra\infty$. Therefore applying Lemma \ref{14}, for every $s\in \mathbb N$ there exists
$k_s\in \mathbb N$ such that for all $n>k_s$ we have
\begin{equation}\label{722}|\langle v_np\dd_\alpha(pu)u^*v_nu,\dd_\alpha(pu)\rangle|\le \frac{1}{s}.\end{equation}

Since $v_n$ and $u$ are unitaries then applying Proposition \ref{70} a few times and using (\ref{723}), (\ref{724}) and (\ref{722}), for all
$\alpha>\max\{\beta^1_s,\beta^2_s\}$ and all $n>k_s$ we have
\begin{eqnarray}\nonumber&&\|\dd_\alpha(pu)\|^2\le 2\|p\dd_\alpha(pu)\|^2+5000\|\dd_\alpha(p)\|\\
\nonumber &=& 2|\langle v_np\dd_\alpha(pu)u^*v_nu,v_np\dd_\alpha(pu)u^*v_nu\rangle|+5000\|\dd_\alpha(p)\|\\
\nonumber&\le&2|\langle v_np\dd_\alpha(pu)u^*v_nu,\dd_\alpha(v_npuu^*v_nu)\rangle|+ 100\|\dd_\alpha(v_np)\|^{\frac{1}{2}}+\\
\nonumber &&+100\|\dd_\alpha(u^*v_nu)\|^{\frac{1}{2}} +5000\|\dd_\alpha(p)\|\\
\nonumber&=&2|\langle v_np\dd_\alpha(pu)u^*v_nu,\dd_\alpha(pu)\rangle|+100\|\dd_\alpha(v_np)\|^{\frac{1}{2}}+100\|\dd_\alpha(u^*v_nu)\|^{\frac{1}{2}} +5000\|\dd_\alpha(p)\|.\\
\nonumber&\le&2|\langle
v_np\dd_\alpha(pu)u^*v_nu,\dd_\alpha(pu)\rangle|+100(\|\dd_\alpha(v_n)\|+50\|\dd_\alpha(p)\|^{\frac{1}{2}})^{\frac{1}{2}}+
\\ \nonumber
&&+100\|\dd_\alpha(u^*v_nu)\|^{\frac{1}{2}} +5000\|\dd_\alpha(p)\|.\\
\nonumber&\le&\frac{2}{s}+ 100(\frac{1}{s}+\frac{50}{s^{\frac{1}{2}}})^{\frac{1}{2}}+\frac{100}{s^{\frac{1}{2}}} +\frac{5000}{s}.
\end{eqnarray}
This shows that $\dd_\alpha$ converges to zero uniformly on $p\mathcal N_N(B)$ and the conclusion follows from a standard averaging
argument.\end{proof}

\vskip 0.1in

For technical reasons that will become apparent in the proof of Lemma \ref{288}, we present here an upgraded version of Lemma \ref{14}. More
precisely, we show the convergence (\ref{17}) still holds if instead of fixed vectors $\xi$ and $\eta$ in $\mathcal K_\pi$ one considers
$\dd_{\alpha}(x_n)$ and $\dd_{\alpha}(y_n)$ for \emph{any} sequences $(x_n)_n,(y_n)_n\in (A\rtimes \Sigma)^{\omega}\vee N$. We note that, in
contrast to Lemma \ref{14}, here we use that $\Sigma$ is a normal subgroup of $\G$ in an essential way.
 \begin{lem}\label{orthog}
Let $\Sigma$ be normal in $\G$, $\pi:\G\ra \UU(\KK)$ be a mixing representation relative to $\Sigma$ and assume that $\G$ admits an unbounded
cocycle into $\KK$ that vanishes on $\Sigma$.  Let $\dd_{\alpha}:N\ra \mathcal H_{\alpha}\otimes_N\mathcal H_{\pi}\otimes_N\mathcal H_{\alpha}
$ be the deformation obtained as before. If $(x_n)_n,(y_n)_n,(c_n)_n,$ $(d_n)_n\in N^{\omega}$ such that $(c_n)_n\perp N(A\rtimes
\Sigma)^{\omega}N$ in $L^2(N^\omega)$ and $(x_n)_n,(y_n)_n\in (A\rtimes \Sigma)^{\omega}\vee N$, then for every $\alpha\ge0$ we have
\begin{equation}\label{13} \lim_{n\ra\omega}\langle c_n\dd_{\alpha}(x_n)d_n,\dd_{\alpha}(y_n)\rangle =0.
\end{equation}
 \end{lem}
\begin{proof} Denote by $C_1=\sup_n \|c_n\|_{\infty}$, $C_2=\sup_n \|d_n\|_{\infty}$, $C_3=\sup_n \|x_n\|_{\infty}$,
 $C_4=\sup_n \|y_n\|_{\infty}$ ($C_{1,2,3,4}<\infty$) and fix an arbitrary $\e>0$. Since $(x_n)_n,(y_n)_n\in (A\rtimes \Sigma)^{\omega}\vee N$
 and $\Sigma$ is normal in $\G$ one can find finite sets $F_1,F_2\subset \G$ of cosets representatives of $\Sigma$ in $\G$ such that
\begin{equation}\label{15}\lim_{n\ra\omega}\|x_n-R_{F_1}(x_n)\|_2<\e \text{ and }\lim_{n\ra\omega}\|y_n-R_{F_2}(y_n)\|_2<\e,\end{equation}
where $R_F$ denotes the projection from $L^2N$ onto the $L^2$-closure of $sp\{au_{\g} \ | \ a\in {A\rtimes \Sigma},\g\in F\}$.

Applying the Cauchy-Schwarz inequality together with $\|m\zeta n\|\le\|m\|_{\infty}\|n\|_{\infty}\|\zeta\|$,
$\|\dd_{\alpha}(m)\|\le\|m\|_2\le\|m\|_{\infty}$ and $\|R_F(m)\|_2\le\|m\|_{\infty}$, for  $m,n\in N; \zeta\in H$) we then have that

\begin{eqnarray}\label{16}&&\\
\nonumber &&\lim_{n\ra\omega}|\langle c_n\dd_{\alpha}(x_n)d_n,\dd_{\alpha}(y_n)\rangle| \le \\
\nonumber &\le& \lim_{n\ra\omega}(|\langle c_n\dd_{\alpha}(x_n-R_{F_1}(x_n))d_n,\dd_{\alpha}(y_n-R_{F_2}(y_n))\rangle|+|\langle c_n\dd_{\alpha}(R_{F_1}(x_n))d_n,\dd_{\alpha}(y_n-R_{F_2}(y_n))\rangle| \\
\nonumber &&+ |\langle c_n\dd_{\alpha}(x_n-R_{F_1}(x_n))d_n,\dd_{\alpha}(R_{F_2}(y_n))\rangle| + |\langle c_n\dd_{\alpha}(R_{F_1}(x_n))d_n,\dd_{\alpha}(R_{F_2}(y_n))\rangle|) \\
\nonumber &\le &C_1C_2\e^2+C_1C_2C_3\e+ C_1C_2C_4\e+ \lim_{n\ra\omega}|\langle
c_n\dd_{\alpha}(R_{F_1}(x_n))d_n,\dd_{\alpha}(R_{F_2}(y_n))\rangle|.
\end{eqnarray}

\noindent Also, by employing the formulas $R_{F_1}(x_n)=\sum_{\g\in F_1} E_{A\rtimes \Sigma}(x_nu^*_{\g})u_{\g}$ and  $R_{F_2}(y_n)=\sum_{\g\in
F_2} E_{A\rtimes \Sigma}(y_nu^*_{\g})u_{\g}$, together with $\dd_{|{A\rtimes \Sigma}}=0$ we have that
\begin{equation}\label{12} |\langle c_n\dd_{\alpha}(R_{F_1}(x_n))d_n,\dd_{\alpha}(R_{F_2}(y_n))\rangle|\le\sum_{\g\in F_1,\l\in F_2} |\langle
E_{A\rtimes \Sigma}(u_{\l}y^*_n)c_n E_{A\rtimes \Sigma}(x_nu^*_{\g})\dd_{\alpha}(u_{\g})d_n,\dd_{\alpha}(u_{\l})\rangle|.\end{equation}

Since $\pi$ is a representation of $\G$ which is mixing relative to $\Sigma$ it follows that $\pi_{\alpha}\otimes \pi\otimes \pi_{\alpha}$ is
also a mixing relative to $\Sigma$. Since the $N$-bimodule $\mathcal H_{\alpha}\overline{\otimes}_N\mathcal H_{\pi}\overline{\otimes}_N\mathcal
H_{\alpha}$ is nothing but the $N$-bimodule coming from the representation $\pi^b_{\alpha}\otimes \pi\otimes \pi^b_{\alpha}$ of $\G$, and since
the normality assumption of $\Sigma$ in $\G$ implies $(E_{A\rtimes \Sigma}(u_{\l}y^*_n)c_n E_{A\rtimes \Sigma}(x_nu^*_{\g}))_n\perp N({A\rtimes
\Sigma})^{\omega}N$ in $L^2(N^\omega)$ for all $\g\in F_1\text{ and } \l\in F_2$ then Lemma~\ref{14} and relation (\ref{12}) give that
$$\lim_{n\ra\omega}|\langle c_n\dd_{\alpha}(R_{F_1}(x_n))d_n,\dd_{\alpha}(R_{F_2}(y_n))\rangle|=0.$$ Combining this with inequality (\ref{16})
we obtain that
\begin{equation*}\lim_{n\ra\omega}|\langle
c_n\dd_{\alpha}(x_n)d_n,\dd_{\alpha}(y_n)\rangle|\le C_1C_2\e^2+C_1C_2C_3\e+ C_1C_2C_4\e. \end{equation*} Since $\e>0$ was arbitrary we
conclude that $\lim_{n\ra\omega}|\langle c_n\dd_{\alpha}(x_n)d_n,\dd_{\alpha}(y_n)\rangle|=0$.\end{proof}

\section{A criterion for uniform convergence of resolvent deformations}\label{sec:convergence}
In this section we exhibit a criterion for uniform convergence of the resolvent deformations on certain subalgebras (Theorem
\ref{derivunifconv}). Roughly speaking, if the resolvent deformation arising from a derivation into a mixing bimodule is small on
``sufficiently large'' sets of elements normalizing a given abelian subalgebra $B$, then $\dd_{\alpha}$ converges uniformly on the unit ball of
$B$. Our proof of the criterion is done in two steps; first, in a technical lemma (Lemma \ref{288} below) we adapt the infinitesimal analysis
developed in \cite{peterson2009} to show that there are infinitely many translations by large projections of the unit ball $(B)_1$
 on which $\dd_{\alpha}$ is uniformly small; then we use a convexity argument to show this implies that $\dd_{\alpha}$ converges uniformly on the
unit ball of $B$ (Theorem \ref{derivunifconv} below).

To introduce the precise statements of the results we first establish the following notation. Let $\Sigma\lhd \G$ be a normal subgroup and let
$\pi:\G\ra \UU(\KK)$ be a representation which is mixing relative to $\Sigma$. Assume there exists an unbounded cocycle $c:\G\ra \KK$ which
vanishes on $\Sigma$. Let $\G \car A$ be a trace preserving action, denote by $M=A\rtimes \G$ and consider $\dd_{\alpha}:M\ra \mathcal
H_{\alpha}\otimes_M\mathcal H_{\pi}\otimes_M\mathcal H_{\alpha} $ corresponding to $c$ as above. Also, assume that $B\subset M=A\rtimes \G$ is
an abelian algebra.

\begin{lem}\label{288}
 Using the above notation, assume there exist an infinite subset $\mathcal F\subset\mathcal N_M(B)$ and a nonzero projection $r\in \mathcal F'\cap M$
 satisfying the following property:

For every $k\in\mathbb N$ there exist $\alpha_k>0$ and a sequence $\{v^k_n \ | \ n\in \mathbb N\}\subset \mathcal F$ such that
\begin{eqnarray}
\label{778}&&\|\dd_{\alpha}(v^k_n)\|_2\le \frac{1}{k}\text{ for all } \alpha>\alpha_k, k, n \in \mathbb N, \\
\label{272} &&\|E_{A\rtimes \Sigma}(xrv^k_ny)\|_2\ra 0 \text{ as }n\ra\infty,\text{ for each }k\in \mathbb N.
\end{eqnarray}

Then one can find a constant $D>0$ such that for every $s\in \mathbb N$ there exist $r_s\in M\text{ with } 0 < r_s\le1$, and positive numbers
$\e_s$, $\b_s>0$, such that for all $\alpha>\b_{s}$ and $b \in \mathcal{U}(B)$, we have
 \begin{eqnarray}\nonumber\|\dd_{\alpha}(r_sb ) \|&\le& \e_s ,\\
\nonumber \tau(rr_s)&\ge&D,\\
\nonumber \e_s&\ra &0\text{ as }s\ra\infty.
\end{eqnarray}
\end{lem}
\begin{proof}
Denoting by $v_k=(v^k_n)_{n\in \mathbb N}\in M^\omega$ we demonstrate the lemma by distinguishing two cases which we will treat separately.

\vskip 0.2in \noindent {\bf Case I}. First we prove the lemma under the assumption
\begin{eqnarray*}
\limsup_k(\sup_{b\in \mathcal{U}(B)}\|E_{(A\rtimes \Sigma)^{\omega}\vee M}(rv_kbv_k^*)-rv_kbv_k^*\|_2)=0.
\end{eqnarray*}
\vskip 0.1in

Observe that by applying Proposition~\ref{70} together with some basic computations we obtain that for all $n,k\in \mathbb N$, $b \in
\mathcal{U}(B)$ and $\alpha>0$ we have following inequality:
\begin{eqnarray*}
 \| \dd_{\alpha}(rb)\|^2 &=& \|v^k_n  \dd_{\alpha}(rb) (v^k_n)^*\|^2\\
& = & \langle v^k_n \dd_{\alpha}(rb) (v^k_n)^*,   v^k_n \dd_{\alpha}(rb) (v^k_n)^*  \rangle \\
& \le &|\langle v^k_n r\dd_{\alpha}(b) (v^k_n)^*, \dd_{\alpha}(v^k_nrb(v^k_n)^*)\rangle|+ 100 \| \dd_{\alpha}(v^k_n)\|^{\frac{1}{2}}+50\|
\dd_{\alpha}(r)\|^{\frac{1}{2}}.
\end{eqnarray*}

Next, for every $k\in\mathbb N$ consider $(y^{k}_n)_n=E_{(A\rtimes \Sigma)^{\omega}\vee M}(v^k_nrb(v^k_n)^*)$. Using the inequality
$\|\dd_{\alpha}(x)\|\le\|x\|_2 \text{ for all }x \in M$, we have
\begin{eqnarray}\label{48}
& & |\langle v^k_nr \dd_{\alpha}(b) (v^k_n)^*, \dd_{\alpha}(v^k_nrb(v^k_n)^*)\rangle| \\
\nonumber & \le & |\langle v^k_nr \dd_{\alpha}(b) (v^k_n)^*,  \dd_{\alpha}(y^{k}_n)\rangle|+ |\langle v_{\l^k_{n}}r\dd_{\alpha}(b)(v^k_n)^*,\dd_{\alpha}(v^k_nrb(v^k_n)^*-y^{k}_n)\rangle| \\
\nonumber & \le &  |\langle v^k_nr \dd_{\alpha}(b) (v^k_n)^*,  \dd_{\alpha}(y^{k}_n)\rangle|+ \|v^k_nrb (v^k_n)^*-y^{k}_n\|_2.
\end{eqnarray}

Fix $s\in \mathbb N$ and notice that since $\limsup_k(\sup_{b\in \mathcal{U}(B)}\|E_{(A\rtimes \Sigma)^{\omega}\vee
M}(v_krbv_k^*)-v_krbv_k^*\|_2)=0$ there exists a natural number $l\ge s$ such that
\begin{equation}\label{348}
\lim_{n\ra\omega}\|v^l_nrb(v^l_n)^*-y^{l}_{n}\|_2<\frac{1}{s}\text{ for all }b\in \mathcal{U}(B).
\end{equation}

Also, since $v^l_n$ satisfies (\ref{272}) we have that $(rv^l_n)_n\perp M(A\rtimes \Sigma)^{\omega}M$. Therefore, since $(y^{l}_n)_n\in
(A\rtimes \Sigma)^{\omega}\vee M$ we can apply Lemma~\ref{orthog} to conclude that $\lim_{n\ra \omega}|\langle rv^l_n \dd_{\alpha}(b)
(v^l_n)^*, \dd_{\alpha}(y^{l}_n)\rangle|= 0$. Finally, using this in combination with (\ref{48}), (\ref{348}) and (\ref{778}) and we obtain
that
\begin{eqnarray*}
\| \dd_{\alpha}(rb)\|^2 &\le &\lim_{n\ra\omega} (|\langle rv^l_n \dd_{\alpha}(b) (v^l_n)^*, \dd_{\alpha}(y^{l}_n)\rangle|  +  \|v^l_nb(v^l_n)^*-y^{l}_n\|_2+100 \|\dd_{\alpha}(v^l_n)\|^{\frac{1}{2}}+50\| \dd_{\alpha}(r)\|^{\frac{1}{2}})\\
&\le& \frac{1}{s}+\frac{100}{l^{\frac{1}{2}}}+\frac{50}{s^{\frac{1}{2}}}\le\frac{1}{s}+\frac{150}{s^{\frac{1}{2}}} \text{ for all
}\alpha>\alpha_s\text{ and } b \in \mathcal U(B).
\end{eqnarray*}

Therefore in this case the conclusion of the lemma follows once we let $r_s=r$, $D=\tau(r)>0$ and
$\e_s=\frac{1}{s}+\frac{150}{s^{\frac{1}{2}}}$.

 To complete the proof it remains to treat the other possibility, therefore we will prove the following:

\vskip0.1in \noindent {\bf Case II.} The conclusion of the lemma holds under the assumption
 \begin{eqnarray*}
 \limsup_k ( \sup_{b\in \mathcal{U}(B)}\|E_{(A\rtimes \Sigma)^{\omega}\vee N}(v_krbv_k^*)-v_krbv_k^*\|_2)>0.
 \end{eqnarray*}

\vskip 0.1in

 Therefore after passing to a subsequence of $v_k$, this
condition implies that there exist $C
> 0$, and a sequences of unitaries $d_k\in \mathcal{U}(B)$ such that $\|E_{(A\rtimes \Sigma)^{\omega}\vee M}(rv_kd_kv_k^*)\|_2<\|r\|_2-C$ for
each $k\in \mathbb N$.

Denote by $c_k=v_kd_kv_k^*$ and notice that if one extends the polar part of $E_{(A\rtimes \Sigma)^{\omega}\vee M}(c_k)$ to a unitary $u_k \in
(A\rtimes \Sigma)^{\omega}\vee M$ then we have that $E_{(A\rtimes \Sigma)^{\omega}\vee M}(u_k^*c_k)=|E_{(A\rtimes \Sigma)^{\omega}\vee
M}(c_k)|$. Let $p_k \in (A\rtimes \Sigma)^{\omega}\vee M$ be the spectral projection of $|E_{(A\rtimes \Sigma)^{\omega}\vee M}(c_k)|$
corresponding to the set $[0, 1-\frac{C}{2})$. This implies $\||E_{(A\rtimes \Sigma)^{\omega}\vee M}(c_k)|(1-p_k)r\|_2\ge
(1-\frac{C}{2})\|(1-p_k)r\|_2$ and using the triangle inequality we obtain

\begin{eqnarray}
\nonumber \|rp_k\|_{2,\omega}&\ge&\|r\|_{2,\omega}-\|r(1-p_k)\|_{2,\omega} \\
\nonumber&\ge& \|r\|_{2,\omega}-\frac{1}{1-\frac{C}{2}}\||E_{(A\rtimes \Sigma)^{\omega}\vee M}(c_k)|(1-p_k)r\|_{2,\omega}\\
\nonumber&\ge& \|r\|_{2,\omega}-\frac{1}{1-\frac{C}{2}}\|E_{(A\rtimes \Sigma)^{\omega}\vee M}(c_k)r\|_{2,\omega}\\
\nonumber&\ge& \|r\|_{2,\omega}-  \frac{\|r\|_2-C}{1-\frac{C}{2}} \ge \frac{C}{2-C}.
\end{eqnarray}

\noindent It then follows that $\frac{C^2}{(2-C)^2}\le\tau_{\omega}(rp_k)$ and notice that we also have
\begin{align*}
 \|p_k|E_{(A\rtimes \Sigma)^{\omega}\vee M}(c_k)|\|_{\infty}&=\| p_kE_{(A\rtimes \Sigma)^{\omega}\vee M}(u_k^*c_k)\|_{\infty} \\
&= \|E_{(A\rtimes \Sigma)^{\omega}\vee M}(u_k^*c_kp_k)\|_{\infty}\le 1-\frac{C}{2}.
\end{align*}

We may assume that $c_k=(c^k_n)_n$ with $c^k_n$ unitaries in $B$, $p_k=(p^k_n)_n$ with $p^k_n$ projections in $M$ such that
$\frac{C}{2-C}\le\tau_{\omega}((p^k_n)_n)$, $u_k^*=((u^k_n)^*)_n$ with $u^k_n$ unitaries in $M$, and $E_{(A\rtimes \Sigma)^{\omega}\vee
M}(u_k^*c_kp_k)=(y^k_n)_n$ with $\|y^k_n\|_{\infty}\le 1-\frac{C}{2}$ for all $n,k$. Since $p_k$ commutes with $|E_{(A\rtimes
\Sigma)^{\omega}\vee M}(c_k)|$ we have $E_{(A\rtimes \Sigma)^{\omega}\vee M}(u_k^*c_kp_k)=E_{(A\rtimes \Sigma)^{\omega}\vee M}(p_ku_k^*c_k)$.
This implies that $p_kE_{(A\rtimes \Sigma)^{\omega}\vee M}(u_k^* c_kp_k)=E_{(A\rtimes \Sigma)^{\omega}\vee M}(u_k^*c_kp_k){p_k}=E_{(A\rtimes
\Sigma)^{\omega}\vee M}(u_k^*c_kp_k)$, thus $\lim_{n\ra\omega}\|y^k_np^k_n-y^k_n\|_2=0$ and $\lim_{n\ra\omega}\|y^k_np^k_n-y^k_n\|_2=0$ for all
$k$. Therefore, by replacing $y^k_n$ with $p^k_ny^k_np^k_n$ we may assume in addition that $p^k_ny^k_n=y^k_np^k_n=y^k_n$ for all $n,k$.

Denote by $z^k_n=p^k_n(u^k_n)^*c^k_n-y^k_n$, $t^k_n=(u^k_n)^*c^k_np^k_n-y^k_n$, and $C' = \frac{4}{C(4-C)} = (1 - (1 - \frac{C}{2})^2 )^{-1}$
and we show next that for all $n,k\in \mathbb N$ and $ b\in \mathcal{U}(B)$ we have the following inequality:
\begin{eqnarray}\label{118}
\|\dd_{\alpha}(p^k_nb) \|^2 &\le & C'(|\langle u^k_nt^k_n\dd_{\alpha}(p^k_nb)(c^k_n)^*, \dd_{\alpha}(b)\rangle| \\
\nonumber &+& |\langle \dd_{\alpha}(p^k_nb), (y^k_n)^*z^k_n \dd_{\alpha}(b) \rangle|+100\|\dd_{\alpha}(p^k_n)\|^{\frac{1}{2}} \\
\nonumber &+& 100 \|\dd_{\alpha}(d_k)\|^{\frac{1}{2}} + 1000\|\dd_{\alpha}(v^k_n)\|^{\frac{1}{4}}).
\end{eqnarray}

Fix $b \in \mathcal{U}(B)$.  Applying Proposition~\ref{70} several times and using $[b,c^k_n]=0$, we obtain the following estimate:
\begin{eqnarray}\label{218}
\|\dd_{\alpha}(p^k_nb) \|^2&=&\langle\dd_{\alpha}(p^k_nb), \dd_{\alpha}(p^k_nb)\rangle\\
\nonumber &\le &|\langle (u^k_n)^*c^k_np^k_n\dd_{\alpha}(p^k_nb), (u^k_n)^*c^k_n\dd_{\alpha}(b_)\rangle|+50\|\dd_{\alpha}(p^k_n)\|^{\frac{1}{2}}\\
\nonumber &\le &|\langle y^k_n\dd_{\alpha}(p^k_nb), (u^k_n)^*c^k_n \dd_{\alpha}(b) \rangle|\\
\nonumber &+& |\langle t^k_n\dd_{\alpha}(p^k_nb), (u^k_n)^*c^k_n\dd_{\alpha}(b) \rangle|+50\|\dd_{\alpha}(p^k_n)\|^{\frac{1}{2}}\\
\nonumber &\le &|\langle y^k_n \dd_{\alpha}(p^k_nb), (u^k_n)^*c^k_n\dd_{\alpha}(b)\rangle|\\
\nonumber & +& |\langle
u^k_nt^k_n\dd_{\alpha}(p^k_nb)(c^k_n)^*,\dd_{\alpha}(b)\rangle|+100\|\dd_{\alpha}(c^k_n)\|^{\frac{1}{2}}+50\|\dd_{\alpha}(p^k_n)\|^{\frac{1}{2}}.
\end{eqnarray}

Using the identity $y^k_np^k_n=p^k_ny^k_n=y^k_n$ and Proposition~\ref{70}, together with inequality $\|y^k_n\|_{\infty}\le 1-\frac{C}{2}$ we
have
\begin{eqnarray}\label{228}
&& |\langle y^k_n\dd_{\alpha}(p^k_nb),(u^k_n)^*c^k_n\dd_{\alpha}(b)\rangle| \\
\nonumber &=& |\langle y^k_n\dd_{\alpha}(p^k_nb), p^k_n(u^k_n)^*c^k_n\dd_{\alpha}(b)\rangle|\\
\nonumber &\le &|\langle y^k_n\dd_{\alpha}(p^k_nb), y^k_n\dd_{\alpha}(b)\rangle|+|\langle y^k_n\dd_{\alpha}(p^k_nb),z^k_n\dd_{\alpha}(b)\rangle|\\
\nonumber &\le &\|y^k_n\dd_{\alpha}(p^k_nb)\|^2+|\langle y^k_n\dd_{\alpha}(p^k_nb), z^k_n \dd_{\alpha}(b)\rangle|+50\|\dd_{\alpha}(p^k_n)\|^{\frac{1}{2}} \\
\nonumber &\le&\|y^k_n\|^2_{\infty}\|\dd_{\alpha}(p^k_nb)\|^2+|\langle \dd_{\alpha}(p^k_nb),(y^k_n)^*z^k_n\dd_{\alpha}(b)]\rangle|+50\|\dd_{\alpha}(p^k_n)\|^{\frac{1}{2}}\\
\nonumber &\le &(1-\frac{C}{2})^2\|\dd_{\alpha}(p^k_nb)\|^2+|\langle
\dd_{\alpha}(p^k_nb),(y^k_n)^*z^k_n\dd_{\alpha}(b)\rangle|+50\|\dd_{\alpha}(p^k_n)\|^{\frac{1}{2}}.
\end{eqnarray}

Finally, Proposition~\ref{70} and inequality $(x+y)^{\frac{1}{2}}\le x^{\frac{1}{2}}+y^{\frac{1}{2}} \text{ for }x,y\ge 0$ imply that
$\|\dd_{\alpha}(c^k_n)\|^{\frac{1}{2}}\le 10\|\dd_{\alpha}(v^k_n)\|^{\frac{1}{4}}+\|\dd_{\alpha}(d_k)\|^{\frac{1}{2}}$.  Combining this with
(\ref{228}) and (\ref{218}) we obtain (\ref{118}).

We will now demonstrate how the above inequality (\ref{118}) implies our lemma.  Fix an arbitrary $k\in \mathbb N$ and observe that since $
(p^k_n)_n\in (A\rtimes \Sigma)^{\omega}\vee M$ and $d_k\in N$ there exists $\b^1_k>0$ such that
\begin{equation}\label{108}
\lim_{n\ra \omega}\|\dd_{\alpha}(p^k_n)\|,\|\dd_{\alpha}(d_k)\|\le \frac{1}{k}\text{ for all }\alpha>\b^1_k.
\end{equation}

Next, since $(u^k_n)_n,((y^k_n)^*)_n\in (A\rtimes \Sigma)^{\omega}\vee M$ we have that $E_{(A\rtimes \Sigma)^{\omega}\vee
M}((u^k_nt^k_n)_n)=E_{(A\rtimes \Sigma)^{\omega}\vee M}(((y^k_n)^*z^k_n)_n)=0$.  Also, it is clear that $(p^k_nb)_n, b = (b)_n \in (A\rtimes
\Sigma)^{\omega}\vee M$ and hence Lemma~\ref{orthog} implies that for all $\alpha>0$ we have
\begin{equation}\label{98}
\lim_{n\ra\omega}|\langle u^k_nt^k_n\dd_{\alpha}(p^k_nb)(c^k_n)^*,\dd_{\alpha}(b)\rangle| = 0,
\end{equation}
and
\begin{equation}\label{998}
\lim_{n \ra \omega} |\langle \dd_{\alpha}(p^k_nb),(y^k_n)^*z^k_n\dd_{\alpha}(b)\rangle| =0.
\end{equation}

Taking $\lim_{n\ra\omega}$ in inequality (\ref{118})(for $k=s$) and using successively (\ref{108}), (\ref{98}), and (\ref{998}) we have that
for all $b\in \mathcal{U}(B)$, and $\alpha> \b^2_s=\max\{\b^1_s,\alpha_s\}$,
\begin{equation*}
\lim_{n\ra\omega}\|\dd_{\alpha}(p^s_nb)\|^2\le C'(\frac{200}{s^{\frac{1}{2}}}+\frac{1000}{s^{\frac{1}{4}}}).
\end{equation*}

Using again Proposition~\ref{70} and (\ref{108}) in conjunction with the previous inequality we have that
\begin{equation}\label{268}
\lim_{n\ra\omega}|\langle p^s_n\dd_{\alpha}(b),\dd_{\alpha}(b)\rangle|\le
C'(\frac{200}{s^{\frac{1}{2}}}+\frac{1000}{s^{\frac{1}{4}}})+\frac{100}{s^\frac{1}{2}},
\end{equation}
for all $\alpha>\b^2_{s}$ and $b\in \mathcal{U}(B)$.

By considering a subsequence, we may assume that (\ref{268}) holds as $ n\ra\infty$ and that $p^s_n$ converges weakly to $x_s\in M$. Since each
$p^s_n$ is a projection satisfying $\frac{C^2}{(2-C)^2}\le\tau(rp^s_n)$ we have that then $0\le x_s\le 1$ and
$\frac{C^2}{(2-C)^2}\le\tau(rx_s)$. Using the weak continuity of the left action of $M$ on $\mathcal H$, inequality (\ref{268}) implies that
for all $\alpha>\b^2_{s}$, $b\in \mathcal{U}(B)$
\begin{equation}\label{278}
\langle x_s\dd_{\alpha}(b),\dd_{\alpha}(b)\rangle\le C'(\frac{200}{s^{\frac{1}{2}}}+\frac{1000}{s^{\frac{1}{4}}})+\frac{100}{s^\frac{1}{2}} .
\end{equation}

Denote by $r_s=x^{\frac{1}{2}}_s$ and let $\b^3_s>0$ such that $\|\dd_{\alpha}(r_s)\|\le \frac{1}{s}$ for all $\alpha>\b^3_s$. Combining this
with (\ref{278}) and Proposition~\ref{70} if we let $\b_s=\max\{\b^3_s,\b^2_s\}$ then we obtain that
\begin{equation*}
\|\dd_{\alpha}(r_sb)\|^2\le C'(\frac{200}{s^{\frac{1}{2}}}+\frac{1000}{s^{\frac{1}{4}}})+\frac{200}{s^\frac{1}{2}},
\end{equation*}
for all  $\alpha>\b_{s}$ and $b\in \mathcal{U}(B)$.

Finally, since for every $s\in \mathbb N$ we have $\frac{C^2}{(2-C)^2}\le\tau(rx_s)\le\tau(rr_s)$, setting $D=\frac{C^2}{(2-C)^2}$ and
$\e_s=[C'(\frac{200}{s^{\frac{1}{2}}}+\frac{1000}{s^{\frac{1}{4}}})+\frac{200}{s^\frac{1}{2}}]^{\frac{1}{2}}$ completes the proof of the lemma.
\end{proof}

Next we use the previous lemma to show the main result of the section.

\begin{thm}\label{derivunifconv}
 Using the above notation, assume there exist an infinite subset $\mathcal F\subset\mathcal N_M(B)$ and a nonzero projection $r\in \mathcal F'\cap M$
 satisfying the following property:

For every $k\in\mathbb N$ there exist $\alpha_k>0$ and a sequence $\{v^k_n \ | \ n\in \mathbb N\}\subset \mathcal F$ such that
\begin{eqnarray}\label{78}
&&\|\dd_{\alpha}(v^k_n)\|_2\le \frac{1}{k}\text{ for all } \alpha>\alpha_k, k, n \in \mathbb N, \\
\label{248} &&\|E_{A\rtimes \Sigma}(xrv^k_ny)\|_2\ra 0 \text{ as }n\ra\infty,\text{ for each }k\in \mathbb N.
\end{eqnarray}
Then there exists a nonzero projection $q\in\mathcal Z(\mathcal N_M(B)'')$ such that $rq\neq0$ and $\dd_{\alpha}$ converges uniformly to zero
on $q(B)_1$ as $\alpha\ra\infty$. In particular, if $B$ is semiregular, then $\dd_{\alpha}$ converges uniformly to zero on the unit ball
$(B)_1$ as $\alpha\ra\infty$.
\end{thm}
\begin{proof} Applying Lemma \ref{288}, there exits a constant $ D>0$ such that for every $s\in \mathbb N$ there exist $r_s\in M\text{ with }
0 < r_s\le1$, and positive numbers $\e_s$, $\b_s>0$, such that for all $\alpha>\b_{s}$ and $b \in \mathcal{U}(B)$, we have
 \begin{eqnarray}\label{328}
 \|\dd_{\alpha}(r_sb ) \|&\le& \e_s ,\\
\label{298}\tau(rr_s)&\ge&D,\\
\label{498}\e_s&\downarrow &0\text{ as }s\ra\infty.
\end{eqnarray}

Next we use these relations in combination with a standard convexity argument to show that $\dd_{\alpha}$ converges to zero uniformly on
$q\mathcal{U}(B)$.

Passing to a subsequence if necessary we can assume that $r_s$ converges weakly to $x$ for some $x\in M$. The above inequalities imply that
$0<x\le 1$ and $\tau(rx)\ge D$. Moreover, for every $s\in \mathbb N$ there exist nonempty, finite set $F_s\subset \mathbb N$, scalars
$\sum_{i\in F_s} \mu^s_i=1$ and positive numbers $\e'_s>0$ such that
\begin{eqnarray}
&&k_s=\min (F_s)\uparrow\infty\text{, as }s\ra \infty, \\
\label{701}&&\|x-\sum_{i\in F_s} \mu^s_ir_i\|_2\le \e'_s,\\
&&\e'_s\uparrow \infty \text{, as }s\ra \infty.
\end{eqnarray}

 Let $\beta^1_s=\max_{l\in F_s}\beta_l$. Therefore, using (\ref{701}) and (\ref{328}), for all $b\in\mathcal U(B)$ and $\alpha>\beta^1_s$ we have the following

\begin{eqnarray}
\label{711}\|\dd_{\alpha}(x b)\|&\le& \|\dd_{\alpha}((x-\sum_{i\in F_s} \mu^s_ir_i) b)\|+\|\dd_{\alpha}((\sum_{i\in F_s} \mu^s_ir_i)b)\|\\
\nonumber&\le& \|x-\sum_{i\in F_s} \mu^s_ir_i\|+\sum_{i\in F_s} \mu^s_i\|\dd_{\alpha}(r_ib)\|\\
\nonumber&\le& \e'_s+\sum_{i\in F_s} \mu^s_i\e_i\\
\nonumber &\le& \e'_s+(\sum_{i\in F_s} \mu^s_i)\e_{k_s}=\e'_s+\e_{k_s}.
\end{eqnarray}

\noindent Also, since $xr\neq 0$, there exists $c>0$ such that if $p$ denotes the spectral projection of $x$ corresponding to the set
${(c,\infty)}$ then we have that $pr\neq 0$. Next, for every $s\in \mathbb N$ we let $\beta^2_s>0$ such that for all $\alpha>\beta^2_s$ we have
\begin{equation}\label{712} \|\dd_{\alpha}(p)\|\le \frac{1}{s},\quad\|\dd_{\alpha}(x)\|\le \frac{1}{s}.
\end{equation}

\noindent Therefore, using Proposition~\ref{70} together with inequality $cp\le x$ and relations (\ref{711})-(\ref{712}), for all $s\in\mathbb
N$, $b\in \mathcal U(B)$ and $\alpha>\max\{\b^1_s,\b^2_s\}$ we have the following

\begin{eqnarray}
\nonumber\|\dd_{\alpha}(p b)\|&\le& \|p\dd_{\alpha}(b)\|+50\|\dd_\alpha(p)\|^{\frac{1}{2}} \\
\nonumber&\le& \frac{1}{c}\|x\dd_\alpha(b)\|+50\|\dd_\alpha(p)\|^{\frac{1}{2}}\\
\nonumber&\le& \frac{1}{c}(\|\dd_\alpha(xb)\|+50\|\dd_\alpha(x)\|^{\frac{1}{2}})+50\|\dd_\alpha(p)\|^{\frac{1}{2}}\\
\label{702} &\le& \frac{1}{c}(\e'_s+\e_{k_s}+\frac{50}{s^{\frac{1}{2}}})+\frac{50}{s^{\frac{1}{2}}}.
\end{eqnarray}
\noindent Notice that since $k_s \uparrow\infty$, $\e_s\ra 0$ and $\e'_s\ra 0$ as $s\ra\infty$ we have that
$\frac{1}{c}(\e'_s+\e_{k_s}+\frac{50}{s^{\frac{1}{2}}})+\frac{50}{s^{\frac{1}{2}}}\ra 0$ as $s\ra\infty$ and hence we conclude that
$\dd_\alpha$ converges uniformly to zero on $p(B)_1$ as $\alpha \ra \infty$.

\vskip 0.1in

Consider the partially ordered set
\begin{eqnarray*}\mathcal V=\{p\in M, \text{ projection } \ | \ \dd_\alpha\text{ converges uniformly to zero on
}q(B)_1 \},
\end{eqnarray*}
where the partial order is given by the regular operatorial order.

\noindent The first part of the proof shows that $p\in \mathcal V$ and therefore $\mathcal F$ is nonempty. Given any increasing chain
$p_\iota\in\mathcal V$ let $p_{\infty}$ to be the supremum of $p_\iota$'s.

Then for every $s\in \mathbb N$ there exists $\iota_s$ such that

\begin{eqnarray}\label{719}\|p_{\iota_s}-p_{\infty}\|_2\le \frac{1}{s}.\end{eqnarray}

Also, since $\dd_\alpha$ converges uniformly to zero on $p_{\iota_s}(B)_1$, for every $s\in \mathbb N$ there exists $\beta^3_s>0$ such that for
all $\alpha>\beta^3_s$ and $b\in (B)_1$ we have
\begin{eqnarray}\label{717}\|\dd_\alpha(p_{\iota_s}b)\|&\le& \frac{1}{s}.
\end{eqnarray}
Therefore, using the triangle inequality together with (\ref{719})-(\ref{717}), for every $s\in \mathbb N$ there exists $\beta^3_s>0$ such that
for all $\alpha>\beta^3_s$ and $b\in (B)_1$ we have
\begin{eqnarray}\label{720} \|\dd_\alpha(p_{\infty}b)\|&\le&\|\dd_\alpha((p_{\infty}-p_{\iota_s})b)\|+\|\dd_\alpha(p_{\iota_s}b)\|\\
\nonumber&\le& \|(p_{\infty}-p_{\iota_s})\|_2+\|\dd_\alpha(p_{\iota_s}b)\| \le \frac{2}{s}.
\end{eqnarray}
This shows that $\dd_\alpha$ converges uniformly to zero on $p_{\infty}(B)_1$.

By Zorn's Lemma, $\mathcal F$ contains a maximal element which we call $q$. Fix $u\in \mathcal N_M(B)$, since $\dd_\alpha$ converges uniformly
to zero on $q(B)_1$ and $u\in \mathcal N_M(B)$ then for every $s\in \mathbb N$ there exists $\beta^4_s>0$ such that for all $\alpha>\beta^4_s$
and $b\in (B)_1$ we have
\begin{eqnarray}\label{713}\|\dd_\alpha(qu^*bu)\|&\le& \frac{1}{s}.
\end{eqnarray}
Also, there exists $\beta^5_s>0$ such that for all $\alpha>\beta^5_s$ we have
\begin{eqnarray}\label{714}\|\dd_\alpha(u)\|&\le& \frac{1}{s}.
\end{eqnarray}
Therefore, using Proposition \ref{70} in combination with (\ref{713}) and (\ref{714}), for every $s\in \mathbb N$,
$\alpha>\max\{\b^5_s,\b^4_s\}$ and $b\in (B)_1$ we have the following

\begin{eqnarray}\nonumber\|\dd_\alpha(uqu^*b)\|&\le&\|u\dd_\alpha(qu^*bu)u^*\|+100\|\dd_\alpha(u)\|^{\frac{1}{2}}\\
\nonumber&=&\|\dd_\alpha(qu^*bu)\|+100\|\dd_\alpha(u)\|^{\frac{1}{2}} \\
\nonumber&\le &\frac{1}{s}+\frac{100}{s^{\frac{1}{2}}}.
\end{eqnarray}

\noindent This shows that $\dd_\alpha$ converges uniformly to zero on $(uqu^*)(B)_1$ and therefore $\dd_\alpha$ converges uniformly to zero on
$(uqu^*+q)(B)_1$. For any $c>0$ denote by $q_c$ the spectral projection of $uqu^*+q$ corresponding to the interval $[c,\infty)$ and let $r_c$
be the inverse of $q_c(uqu^*+q)$ in $q_cMq_c$. Since $\dd_\alpha$ converges uniformly to zero on $(uqu^*+q)(B)_1$ then by using Proposition
\ref{70} again we obtain that $\dd_\alpha$ converges to zero uniformly on $r_c(uqu^*+q)(B)_1=q_c(B)_1$. Since this holds for every $c>0$ we
conclude that $\dd_\alpha$ converges to zero uniformly on $q_0(B)_1$ where $q_0=supp(uqu^*+q)$

Since $supp(vuqu^*v^*+q)=uqu^*\vee q$ and since $q$ was a maximal element of $\mathcal V$ we have that $q=uqu^*\vee q$, or equivalently,
$q=uqu^*$. Since the above procedure can be done for any $u\in \mathcal N_M(B)$ we have that $q\in \mathcal N_M(B)'\cap M$ and thus $q\in
\mathcal Z(\mathcal N_M(B)'')$. Also, notice that a similar argument as above shows that $q$ is the unique maximal element of $\mathcal V$.
Combining this with the first part of the proof we have that $q\ge p$ and therefore $qr\neq 0$.\end{proof}

\section{A Transfer Lemma}\label{sec:transfer}
For the transfer lemma that we cover in this section we need some preliminary results on Hilbert bimodules. The following lemma is well known
to the experts and we include a proof only for completeness.

\begin{lem}\label{lem:diagonalbimodule}
Suppose that $B$ is a diffuse amenable von Neumann algebra, $\La$ is a countable discrete group, and $\La \car B$ is a trace preserving action
such that $N = B\rtimes \La$.  Let $\mathcal H$ be a Hilbert $N$-bimodule.  Let $\Delta:N\ra N\overline{\otimes}N$ be the $*$-homomorphism
defined by $\Delta(\sum b_{\l}v_{\l})=\sum b_{\l}v_{\l}\otimes v_{\l}$  where $\l\in\La$ and $b_{\l}\in B$, and consider the $N$-bimodule
$L^2N\oo \mathcal H$ where the left-right $N$-actions satisfy $x\cdot \xi \cdot y =\Delta(x)\xi\Delta(y)$ for all $x,y\in N, \xi\in L^2N\oo
\mathcal H$. If the $N$-bimodule $\mathcal H$ is weakly contained in the coarse bimodule then the $N$-bimodule $L^2N\oo \mathcal H$ is also
weakly contained in the coarse bimodule.
\end{lem}

\begin{proof}
Since, as $N$-bimodules, ${\HH}$ is weakly contained in ${L^2N \oo L^2N}$ it follows that, as $N\oo N$-bimodules, ${L^2N \oo \HH}$ is weakly
contained in $ L^2N \oo (L^2N \oo L^2N)$ and hence it is enough to consider the case when $\HH = L^2N \oo L^2N$.

In this case if we consider an orthonormal basis $\{ b_i \}_{i \in I}$ for $L^2B$, and consider the vectors $\eta = 1 \otimes (b_i v_h \otimes
b_j)$ and $\zeta = 1 \otimes (b_k v_{h'} \otimes b_l)$ then a routine calculation shows that
$$
\langle b v_{\l} \cdot \eta \cdot b' v_{\l'}, \zeta \rangle = \dd_{\l, e}\dd_{\l', e} \dd_{j, l} \dd_{h, h'} \dd_{i, k} \tau(b b' )  = \tau(
E_B(b v_\l) b' v_{\l'})
$$
This then establishes a Hilbert $N$-bimodule isomorphism between ${L^2N \oo (L^2N \oo L^2N)}$ and ${L^2\langle N, e_B \rangle^{\oplus
\infty}}$.  Since $B$ is amenable the previous bimodule is weakly contained in the coarse bimodule.
\end{proof}

For the rest of this section we will use the following notation.

{\bf Notation.} For $1\le i\le 2$ let $\G_i\in \mathcal{CR}$ and let $\Omega_i\subset \G_i$ be an i.c.c.\ subgroup with relative property (T).
Denote by $ \G = \G_1 \times \G_2$ and $\Omega=\Omega_1 \times \Omega_2 $ and assume that $\G \car A$ is a trace preserving action on an
abelian von Neumann algebra $A$ such that $N = A \rtimes \G$ is a factor. Suppose that $B$ is an abelian algebra, $\Lambda$ is an i.c.c.\ group
and $\La \car^\rho B$ a free action such that $N = A \rtimes \G = B \rtimes_\rho \La$. Let $\Delta: N \to N \oo N$ be the $*$-homomorphism from
the previous lemma defined by $\Delta(\Sigma_{\l \in \La} b_\l v_\l ) = \Sigma_{\l \in \La} b_\l v_\l \otimes v_\l$ for $\l \in \La$, and $b_\l
\in B$. Also, as in Section \ref{sec:intertwining}, for every $i$ we denote by $\G_{(i)}$ the subgroup of $\G$ consisting of all elements in
$\G$ whose $i^{\text{th}}$ coordinate is trivial.

Since $\G_i \in \mathcal{CR}$ for all $1 \le i \le 2$ there exists a corresponding unbounded cocycle into a mixing representation which is
weakly contained in the left-regular representation. As explained in Section~\ref{sec:derivations}, associated to such a cocycle is a closable
real derivation $\dd^i: N \to \HH_i$, such that the $N$-bimodule ${\HH_i}$ is mixing relative to $A \rtimes \G_{(i)}$, and such that
$\dd^i_\alpha$ does not converge uniformly on $(N)_1$.  Also, we denote by $\hat \dd^i = 0 \otimes \dd^i: N \oo N \to L^2N \oo \HH_i$ the
tensor product derivation as described in Section~\ref{sec:derivations}.

Next we show two non-intertwining lemmas which will be very important in establishing the main result of this section.

\begin{lem}\label{lem:nonintertwining} Using the notation above, for all $1\le j\le 2$ and $q_j\in \Delta(L\Omega_j)'\cap( N\oo N)$ nonzero projections
we have that $\Delta(L\Omega_j)q_j\nprec_{N\oo N }N\oo A$.
\end{lem}
\begin{proof} We will proceed by contradiction. So assume there exits $1\le i\le 2$ and $q_i\in \Delta(L\Omega_i)'\cap( N\oo N)$ such that $\Delta(L\Omega_i)q_i\prec_{N\oo N }N\oo A$.
Since $A$ is abelian then $N\oo A $ has the relative Haagerup's property with respect to $N \oo 1$ and therefore by Lemma 1 in
\cite{houdayer-popa-vaes-2010} we have that

\begin{equation*}\label{401}\Delta(L\Omega_i)q_i\prec_{N\oo N }N\oo 1.\end{equation*}

\noindent Then by Lemma 9.2 (i) in \cite{adi2010} this would further imply $(L\Omega_i )q_i \prec_{N} B$, which is obviously impossible because
$L\Omega$ is a nonamenable factor while $B$ is abelian.\end{proof}

\vskip 0.01in

\begin{lem}\label{lem:nonconverge}
Using the notation above, let $1\le j\le 2$ and $p\in \Delta(L\Omega)'\cap( N\oo N)$ a nonzero projection  such that
$\Delta(L\Omega)^{\omega}p\subset (N\oo(A\rtimes \G_{j}))^{\omega}\rtimes \G_{(j)}$. Then for every $1\le k\le 2$ we have
$\Delta(L\Omega_k)p\nprec_{N\oo N}N\oo(A\rtimes \G_{(j)})$; in particular $\Delta(L\Omega)p\nprec_{N\oo N}N\oo(A\rtimes \G_{(j)})$.
\end{lem}

\begin{proof}

We will proceed by contradiction. Assuming that $\Delta(L\Omega_k)p\prec_{N\oo N}N\oo(A\rtimes \G_{(j)})$, by Proposition \ref{intertwining},
there exists $p'$ a nonzero sub-projection of $p$ such that $\Delta(L\Omega)^{\omega}p'\subset (N\oo(A\rtimes \G_{(j)}))^{\omega}\rtimes
\G_{j}$. Combining this with the hypothesis assumption we obtain
\begin{equation*}
\Delta(L\Omega_k)^{\omega}p'\subset [(N\oo(A\rtimes \G_{(j)}))^{\omega}\rtimes \G_{j}]\cap[(N\oo(A\rtimes \G_{j}))^{\omega}\rtimes \G_{(j)}],
\end{equation*}
and hence $\Delta(L\Omega)^{\omega}p'\subset [(N\oo A)^{\omega}\rtimes (\G_1\times \G_2)$. By Proposition \ref{intertwining} again this further
implies that $\Delta(L\Omega_k)^{\omega}p'\prec_{N\oo N}N\oo A$. This however contradicts Lemma \ref{lem:nonconverge} and we are
done.\end{proof}

\begin{lem}\label{lem:unifconv} Using the notation above, if $\hat\dd^j_{\alpha}$ converges to zero uniformly on $(N\oo L\La)_1$ then $(N\oo L\La)^{\omega}\subset (N\oo(A\rtimes \G_{(j)}))^{\omega}\rtimes
\G_{j}$.
\end{lem}
\begin{proof} By proof of Corollary \ref{reducingintertwininginproducts}, to get our conclusion it suffices to prove that for every $q\in (N\oo L\La)'\cap (N\oo N)= \mathcal Z(N\oo L\La)$ we have
$(N\oo L\La)q\prec_{N\oo N} N\oo(A\rtimes \G_{(j)})$.

To show this we proceed by contradiction so assume that there exists $q_o\in \mathcal Z(N\oo L\La)$ such that $(N\oo L\La)q_o\nprec_{N\oo N}
N\oo(A\rtimes \G_{(j)})$. Also since $\hat\dd^j_{\alpha}$ converges to zero uniformly on $(N\oo L\La)_1$ then applying Corollary 2.3 in
\cite{popa2004}, for every $k\in \mathbb N$ there exists $\alpha_k>0$ and an infinite sequence of elements $\{u^k_n\ | \ n\in \mathbb
N\}\subset \mathcal U(N\oo L\La)$ such that
\begin{enumerate}
\item $\|\hat \dd^j_\alpha(u^k_n)\|\le\frac{1}{k}\text{ for all } \alpha>\alpha_k;$
\item $\|E_{N\oo(A\rtimes \G_{(j)})}(x u^k_nq_oy)\|_2\ra 0\text{ as } n\ra\infty \text{ for all }x,y\in N\oo N.$
\end{enumerate}
\noindent Then applying Theorem \ref{derivunifconv} for $\G=\G_j$, $\Sigma$ trivial, $M=N\oo N=(N\oo (A\rtimes \G_{(j)}))\rtimes \G_j$, $B=1\oo
B $, and $r=q_o$ we obtain that $\hat \dd^j_\alpha$ converges to zero uniformly on $(1 \oo B)_1$. Applying Proposition \ref{70} we obtain that
$\hat \dd^j_\alpha$ converges to zero uniformly on $(\mathcal U( 1 \oo B)\mathcal U(N\oo L\La))$ and hence on $(N\oo N)_1$ which is obviously a
contradiction.
\end{proof}

\begin{lem}\label{lem:notcontained}
Using the notation above assume that for every $1\le j\le 2$ let $p_j\in (N\oo L\La)'\cap (N\oo N)=1\oo (L\La'\cap N) $ be the maximal
projection such that $\hat\dd^j_{\alpha}$ converges to zero uniformly on $p_j(N\oo L\La)_1$. Then for every $1\le j\le 2$ we have that $p_j\neq
1$ and $\Delta(L\Omega)(1-p_j) \nprec_{N \oo N} N \oo (A \rtimes \G_{j})$.
\end{lem}

\begin{proof} First, we notice that if relations $p_1 =1$ and $p_2=1$ would hold simultaneously then by the previous lemma we have that $(N\oo L\La)^{\omega}\subset (N\oo(A\rtimes \G_{(j)}))^{\omega}\rtimes
\G_{j}$ for every $1\le j\le 2$. This would further imply that
\begin{eqnarray*}(N\oo L\La)^{\omega}\subset [(N\oo(A\rtimes \G_{1}))^{\omega}\rtimes
\G_{2}]\cap[(N\oo(A\rtimes \G_{2}))^{\omega}\rtimes \G_{1}]=(N\oo A)^{\omega}\rtimes (\G_1\times \G_{2}),\end{eqnarray*}

\noindent which by Proposition \ref{intertwining} gives that $N\oo L\La\prec_{N\oo N} N\oo A$. This however would contradict Lemma
\ref{lem:nonintertwining}. Hence we can assume without loss of generality that $p_1\neq 1$, or equivalently, $0<1-p_1$.

Next, we proceed by contradiction to show that $\Delta(L\Omega)(1-p_1) \nprec_{N \oo N} N \oo (A \rtimes \G_{1})$. So assume that
$\Delta(L\Omega)(1-p_1) \prec_{N \oo N} N \oo (A \rtimes \G_{1})$ and by Proposition \ref{intertwining} there exists a non-zero projection
$r_1\in\Delta(L\Omega)'\cap N\oo N$ with $r_1\le 1-p_1$ such that
\begin{eqnarray*} \Delta(L\Omega)^{\omega}r_1 \subset (N\oo (A \rtimes \G_{1}))^{\omega}\rtimes \G_{2}. \end{eqnarray*}

Since the pair $(\G, \Omega)$ has relative property (T) then the inclusion $\Delta(L\Omega)\subset N\oo N$ is rigid and therefore $\hat
\dd^1_\alpha$ converge uniformly to zero on $(\Delta(L\Omega))_1$. Also by Lemma \ref{lem:nonconverge} we have that $\Delta(L\Omega)r_1
\nprec_{N \oo N} N \oo (A \rtimes \G_{2})$ and applying Corollary 2.3 in \cite{popa2004}, for every $k\in \mathbb N$ there exists $\alpha_k>0$
and an infinite sequence of elements $\{\l^k_n|n\in \mathbb N\}\subset \Omega$ such that
\begin{enumerate}
\item $\|\hat \dd^1_\alpha(\Delta(u_{\l^k_n}))\|\le\frac{1}{k}\text{ for all } \alpha>\alpha_k;$
\item $\|E_{N\oo(A\rtimes \G_{2})}(x\Delta(u_{\l^k_n})r_1y)\|_2\ra 0\text{ as } n\ra\infty \text{ for all }x,y\in N\oo N.$
\end{enumerate}

\noindent Then applying Theorem \ref{derivunifconv} for $\G=\G_1$, $\Sigma$ trivial, $M=N\oo N=(N\oo (A\rtimes \G_{2}))\rtimes \G_1$,
$B=\Delta(A)$, and $r=r_1$ there exists a nonzero projection $q_1\in \mathcal Z(\mathcal N_{N\oo N }(\Delta(A))'') $ with $r_1q_1\neq 0$ such
that $\hat \dd^1_\alpha$ converges to zero uniformly on $q_1(\Delta(A))_1$.

Next, we claim that $\hat\dd^1_{\alpha}$ converges to zero uniformly on $r_1(\Delta(L\G_j))_1$ for every $1\le j\le 2$. For the proof just
notice that since the pair $(\G_j, \Omega_j)$ has relative property (T) the inclusion $(\Delta(L\Omega_j)\subset N\oo N)$ is rigid and
$\hat\dd^1_{\alpha}$ converges uniformly to zero on $(\Delta(L\Omega_j))_1$. Also, by Lemma \ref{lem:nonconverge} for all $1\le j\le 2$ we have
$\Delta(L\Omega_j)r_1 \nprec_{N \oo N} N \oo (A \rtimes \G_{2})$ and therefore the conclusion follows from Theorem \ref{unifconvonnormalizer}.

Moreover, applying the same maximality argument as in the proof of Theorem 3.2, one can find a projection $r'\in \mathcal Z(\Delta(L\G)'\cap
N\oo N)$ with $r'q_1\neq0$ such that $\hat\dd^1_{\alpha}$ converges to zero uniformly on $r'(\Delta(L\G_j))_1$ for every $1\le j \le 2$.
 Since $r'$ commutes with $\Delta (L\G_j)$ for every $1\le j\le 2$, using Proposition \ref{70} we obtain that $\hat\dd^i_{\alpha}$ converges
to zero uniformly on $r'(\mathcal U(\Delta(L\G_1)\mathcal U(\Delta(L\G_2))$ and therefore on $r'((\Delta(L\G))_1$.

Since $\hat\dd^1_{\alpha}$ converges uniformly to zero on $q(\Delta(A))_1$  and $r'$ commutes with $q_1$ then applying Proposition \ref{70} one
more time we obtain that $\hat \dd^1_\alpha$ converges uniformly on $r'q_1(\mathcal U(\Delta(L\G))\mathcal U(\Delta(A)))$ and hence on
$r'q_1(\Delta(N))_1$.

Also, by definition, we have that $\hat \dd^1_\alpha$ converges uniformly to zero on $(N \oo 1)_1$ and therefore by Proposition \ref{70} we
conclude that $\hat \dd^1_\alpha$ converges to zero uniformly on $r'q_1(N\oo L\La)_1$. Since $r'\le 1-p_1$ we have that $0<r'q_1\le 1-p_1$.
Applying the same maximality argument from the proof of Theorem \ref{derivunifconv} one can find a projection $r''\in \mathcal (N\oo L\La)'\cap
N\oo N $ with $r''\le1-p_1$ such that $\hat \dd^1_\alpha$ converges to zero uniformly on $r''(N\oo L\La)_1$. Altogether this gives that $\hat
\dd^1_\alpha$ converges to zero uniformly on $(p_1+r'')(N\oo L\La)_1$ which contradicts the maximality of $p_1$. Therefore we have that
$\Delta(L\Omega)(1-p_1) \nprec_{N \oo N} N \oo (A \rtimes \G_{1})$.

Notice that by Proposition \ref{intertwining} this implies that $(\Delta (L\Omega))^{\omega}(1-p_1)\nsubseteq (N\oo(A\rtimes
\G_{1}))^{\omega}\rtimes \G_{2}$ and hence $(N\oo L\La)^{\omega}\nsubseteq (N\oo(A\rtimes \G_{1}))^{\omega}\rtimes \G_{2}$. By Proposition
\ref{lem:unifconv} this further implies that $\hat \dd^{2}_\alpha$ does not converge to zero uniformly on $(N\oo L\La)_1$ and hence $p_{2}\neq
1$. Proceeding as above we then have that $\Delta(L\Omega)(1-p_{2}) \nprec_{N \oo N} N \oo (A \rtimes \G_{2})$.\end{proof}

These results on Hilbert bimodules and intertwining are used to prove a ``transfer lemma'' \'{a} la Popa-Vaes (Lemma 3.2 in \cite{popavaes09})
that will be of essential use in the proof of this paper's main result. Roughly speaking, the lemma states that the presence of a rigid part on
the source group can be transferred, at the level of resolvent deformations, to ``large'' subsets of the target group. The proof is essentially
the same as in \cite{popavaes09}, the main difference being that here we use Lemma~\ref{lem:notcontained} in place of
Lemma~\ref{lem:diagonalbimodule} used by Popa and Vaes.

\begin{lem}\label{transfer} Let $\G_1,\G_2 \in \mathcal{CR}$, $\G = \G_1 \times
\G_2$ and let $\G \car A$ a trace preserving action on an abelian von Neumann algebra $A $. Let $B$ be an abelian algebra and $\La \car B$ be a
free action such that $N = A \rtimes \G = B \rtimes \La$. Then for every $1 \le j \le 2$, $k\in \mathbb N $, there exists $\alpha_k>0$, and an
infinite set of elements $S_k\subset \La$, such that
\begin{eqnarray}
\label{6}&&\|\dd^i_{\alpha}(v_{\l})\|\le \frac{1}{k} \text{ for all }\alpha>\alpha_k, \l\in S, 1 \le i \le 2; \\
\label{5} &&\|E_{A\rtimes \G_{j}}(xv_{\l}(1-p_j)y)\|_2 \ra 0 \text{ for all } x,y\in N \text{ as }\l\ra\infty,
\end{eqnarray}
where $p_j$ is defined as in Lemma \ref{lem:notcontained}.
\end{lem}

\begin{proof}

Fix $1 \le j \le 2$ and $k\in \mathbb N > 0$. Since the pair $(\G, \Omega)$ has relative property (T) then we also have that the inclusion
$(L\Omega \subset L\G)$ is rigid, and hence so is the inclusion $(\Delta(L\Omega) \subset N \oo N)$.

Since $(\Delta(L\Omega) \subset N \oo N)$ is rigid there exists $\alpha_k > 0$ such that for all $\alpha > \alpha_k$ we have
\begin{equation}\label{1}
\| \hat \dd^i_\alpha \circ \Delta(x) \| \le \frac{1}{k\sqrt{8}}, {\rm \ for \ all \ } x \in (L\Omega)_1 \text{ and } 1\le i\le 2.
\end{equation}

Fix $w\in L\Omega$ a unitary and let $w=\sum_{\l\in \La} w_{\l}v_{\l}$ with $w_{\l}\in B$ be its Fourier expansion in $B\rtimes \La=N$.  If for
every $1\le i\le 2$ we denote by $S^i_k=\{\l\in \La \ | \ \|\dd_{\alpha}^i (v_{\l})\|_2\le \frac{1}{k} \}$, then continuity of $\hat
\dd^i_{\alpha}$, together with (\ref{1}), implies that for all $\alpha > \alpha_k$ we have
\begin{eqnarray*}
\frac{1}{8k^2}&\ge & \| \hat \dd^i_{\alpha}\circ \Delta(w)\|^2\\
&=& \|\hat \dd^i_{\alpha}(\sum_{\l\in \La}(w_{\l}v_{\l}\otimes v_{\l} ))\|^2\\
&=& \|\sum_{\l\in \La}(w_{\l}v_{\l})\otimes \dd^i_{\alpha}(v_{\l}) \|^2\\
&=& \sum_{\l\in \La}\|\dd^i_{\alpha}(v_{\l})\|^2\|w_{\l}\|^2_2\\
&\ge& \frac{1}{k^2}\sum_{\l\in \La\setminus S^i_k}\|w_{\l}\|^2_2.
\end{eqnarray*}
Therefore, we have that $\sum_{\l\in \La\setminus S^i_k}\|w_{\l}\|^2_2\le \frac{1}{8}$ and since $w$ is a unitary, if we denote by $S_k =
S^1_k\cap S^2_k$ then we conclude that for all $w\in \mathcal U(L\Omega)$ we have
\begin{equation}\label{3}
\sum_{\l\in S_k}\|w_{\l}\|^2_2\ge \frac{3}{4}.
\end{equation}

If (\ref{5}) does not hold for this $S_k$ then there exists a finite subset $F\subset N$ and $c>0$ such that
\begin{equation}\label{2}
\sum_{z,y\in F}\|E_{A\rtimes \G_{j}}(z^*v_{\l}(1-p_j)y)\|^2_2\ge c\text{ for all }\l\in S_k.
\end{equation}

Let $\mathcal L=L^2N\overline{\otimes} L^2\langle N,e_{A\rtimes\G_{j}} \rangle$ and consider the vector $\xi=\sum_{z\in F} 1\otimes
(z^*e_{A\rtimes \G_{j}}z ) \in \mathcal L$.

Using relations (\ref{3}) and (\ref{2}) we then obtain the following inequalities
\begin{eqnarray*}
\langle \Delta(w)(1-p_j)\xi (1-p_j)\Delta(w^*),\xi\rangle &=& \sum_{\l\in \La}\|w_{\l}\|^2_2(\sum_{z,y\in F}\|E_{A \rtimes \G_{(j)}}(z^*v_{\l}y)(1-p_j)\|^2_2)\\
&\ge & c\sum_{\l\in S_k }\|w_{\l}\|^2_2\ge \frac{3c}{4} \text{ for all }w\in\mathcal{U}(L\Omega).
\end{eqnarray*}

This implies that the unique $\|\cdot\|_{\tau \times Tr}$-minimal vector in the convex hull of $\{\Delta(w)(1-p_j)\xi (1-p_j)\Delta(w^*) \ | \
w\in \mathcal{U}(L\Omega)\}$ is nonzero and $\Delta(L\Omega)(1-p_j)$-central. However this contradicts Lemma~\ref{lem:notcontained}.
\end{proof}

\section{Uniqueness of group measure space Cartan subalgebras}\label{sec:cartan}

In this section we use the above transfer lemmas in combination with the criterion for uniform convergence of the resolvent deformations
 to prove unique group-measure space Cartan results.  More precisely, our main result shows that
any free, ergodic action of any product groups belonging to  $\mathcal{CR}$ gives rise to a von Neumann algebra with a unique group measure
space Cartan subalgebra.

Exploiting techniques from \cite{peterson2009} we prove first the following key intertwining result.

\begin{thm}\label{mainintertwining}
Let $\G_1, \G_2 \in \mathcal{CR}$, $\G = \G_1 \times \G_2$ and let $\G \car A$ a trace preserving action on an abelian von Neumann algebra $A
$. If we assume that $B$ is an abelian von Neumann algebra and $\La \car B$ a free, action such that $N = A \rtimes \G = B \rtimes \La$ then we
have $B \prec_N A$.

\end{thm}

\begin{proof}
 Fixing $1\le i\le 2 $ notice that the algebra $N$ can be seen as $(A\rtimes
\G_{(i)})\rtimes \G_i$ with $\G_i$ acting trivially on $\G_{(i)}$. Therefore applying the transfer Lemma \ref{transfer} together with
Theorem~\ref{derivunifconv} (for $\Sigma=\{e\}$)  we have that $\dd_\alpha^i$ converges to zero uniformly on $(B)_1$. Next, we proceed by
contradiction to show that $B\prec_N  A \rtimes \G_{(i)}$. Assuming $B\nprec_N A\rtimes \G_{(i)}$, by Popa's intertwining techniques (see
Corollary 2.3. in \cite{popa2004}) there exists a sequence of unitaries $b_n\in \mathcal{U}(B)$ such that $\|E_{A\rtimes
\G_{(i)}}(xb_ny)\|_2\ra 0$ as $n\ra \infty$. Since $\G_{i} \in\mathcal{CR}$ and $\dd^{i}_{\alpha}$ converges uniformly to zero on $(B)_1$,
Theorem \ref{unifconvonnormalizer} implies  that $\dd^{i}_{\alpha}$ converges to zero uniformly on the unit ball of $\mathcal N_M(B)''=N$ which
is obviously a contradiction. Hence for all $1 \le i \le 2$ we have that $B \prec_N A \rtimes \G_{(i)}$ and since $B$ is a Cartan subalgebra of
$N$ then Corollary~\ref{reducingintertwininginproducts} implies that $B \prec_N A$.
\end{proof}

An immediate consequence of the previous theorem is the following:

\begin{cor}
If $\G_1, \G_2 \in \mathcal{CR}$ and $\G = \G_1 \times \G_2$  then the group von Neumann algebra $L\G$ cannot be decomposed as a crossed
product $L\G=B\rtimes \La$, where $\La\car B$ is a free action on a diffuse, abelian von Neumann algebra $B$.
\end{cor}

\begin{proof}
If we assume that $L\G = B\rtimes \La$ then applying the previous result for $A=\mathbb C 1$ we would have $B\prec_N \mathbb C 1 $ which is
obviously a contradiction because $B$ is diffuse.
\end{proof}

Theorem~\ref{mainintertwining} can also be used to obtain von Neumann algebras with unique group measure space Cartan subalgebra.

\begin{cor}\label{cor:uniquegmscartan}

Let $\G_1, \G_2 \in \mathcal{CR}$, $\G = \G_1 \times \G_2$ and let $\G \car X$ be a free measure preserving action on a standard probability
space. If there exists $\La \car Y$ a free measure preserving action on a standard probability space such that $N = L^{\infty} (X)\rtimes \G =
L^{\infty}(Y) \rtimes \La$ then one can find a unitary $u\in N$ such that we have $uL^{\infty}(Y) u^* = L^{\infty} (X)$.

\end{cor}

\begin{proof}
By Theorem~\ref{mainintertwining} we have $L^{\infty}(Y)\prec_N  L^{\infty} (X)$.  Popa's conjugacy criterion (Theorem \ref{intertwining-conj})
for Cartan subalgebras then gives the desired conclusion.
\end{proof}


\section{$W^*$-superrigidity applications}\label{sec:applications}

In this section we use the technical results from previous section to manufacture new examples of $W^*$-superrigid actions. By definition, an
action $\G\car X$ is called $W^*$-superrigid if, for every free, p.m.p.\ action $\La\car Y$, an isomorphism between the crossed products von
Neumann algebras $L^{\infty}(X)\rtimes\G$ and $L^{\infty}(Y)\rtimes\La$ entails conjugacy of the actions $\G\car X$ and $\La\car Y$.

As explained in the introduction, the strategy to produce such actions is to find OE-superrigid actions that give rise to von Neumann algebras
with unique group measure space Cartan subalgebras. Using his influential \emph{deformation/rigidity} theory, Popa discovered the following
class of OE-superrigid actions of product groups.

\begin{thm*}[Corollary 1.3 in \cite{popa07b}] For $i=1,2$, let $\G_i$ be nonamenable groups such that $\G=\G_1\times \G_2$ has no normal finite
subgroup and let $\G\car^{\sigma} X$ be a free, p.m.p.\ $s$-malleable action (see \cite{popa2006} for the definition of an $s$-malleable
action).  If we assume that the restriction $\G_1\car^{\sigma_{|\G_1}}(X,\mu)$ is weak mixing and $\G_2\car^{\sigma_{|\G_2}}(X,\mu)$ has stable
spectral gap then $\G\car^{\sigma} X$ is OE-superrigid.
\end{thm*}
When this result is combined with Corollary~\ref{cor:uniquegmscartan} we obtain the following $W^*$-superrigidity statement:

\begin{cor}\label{w-superrigidity} Let $\G_1,\G_2 \in \mathcal{CR}$ such that $\G = \G_1\times \G_2$ has no normal finite subgroup and let
$\G\car^{\sigma} X$ be a free, p.m.p.\ $s$-malleable action. If we assume that the restricted action $\G_1\car^{\sigma_{|\G_1}}(X,\mu)$ is weak
mixing and $\G_2\car^{\sigma_{|\G_2}}(X,\mu)$ has stable spectral gap then the action $\G\car^{\sigma} X$ is $W^*$-superrigid.
\end{cor}

We mention the following concrete examples of $W^*$-superrigid actions arising from generalized Bernoulli actions. Consider $\G=\G_1\times\G_2$
 where $\G_i\in\mathcal{CR}$ and $\G$ has no normal finite subgroups, let $I$ be a countable, faithful $\G$-set such that for all $i\in I$
 the orbit $(\G_1)i$ is infinite and the stabilizer $\{\g\in \G_2 \ |  \g i=i\}$ is amenable. From the proof of Lemma 3.3 in
 \cite{popa07b} the generalized Bernoulli action $\G\car(\mathbb T,\lambda)^I$ satisfies the conditions in the hypothesis of the previous corollary
  and thus it follows $W^*$-superrigid. \vskip 0.1in
  Monod and Shalom unveiled in \cite{monodshalom} a family of actions of certain product groups that are very close to being OE-superrigid
  (See also Theorem 44 in \cite{sako2010} for additional examples).  For a better understanding of their result
  we need to introduce some terminology.

An action $\G\car(X,\mu)$ is called \emph{mildly mixing} if whenever $A\subseteq X$ and $\g_n\in\G $ then $ \mu(\g_nA\triangle A)\ra 0\text{ as
}\g_n\ra \infty$ only when $A$ is either null or conull.  It is not hard to see that mixing implies mildly mixing, which in turn implies weak
mixing. Also an action $\G\car X$ is called \emph{aperiodic} if its restriction to any finite index subgroup of $\G$ is ergodic.

Following the notations in \cite{monodshalom}, one says that a group $\G$ belongs to $\mathcal C_{\text{reg}}$ if it has nonvanishing second
bounded cohomology with coefficients in the left regular representation, i.e., $H^2_b(\G,\ell^2(\G))\neq 0$. The class $\mathcal
C_{\text{reg}}$ is fairly rich, including all groups which admit a non-elementary nonsimplicial action on a simplicial tree, which is proper on
the set of edges; and all groups which admit a non-elementary proper isometric action on some CAT(-1)-space (see \cite{monodshalom} for more
examples). In particular, any non-elementary, amalgamated free product $\G_1\ast_{\Omega}\G_2$ belongs to $\mathcal C_{\text{reg}}$ if one
assumes that the subgroup $\Omega$ is almost malnormal in one of the factors (Corollary 7.10 in \cite{monodshalom1}).

  Using second bounded cohomology methods, Monod and Shalom proved the following OE-strong rigidity result:
\begin{thm*}[Theorem 1.10 in \cite{monodshalom}]\label{oestrogrigidity} Let $\G=\G_1\times\G_2$ where $\G_i\in \mathcal C_{\text{reg}}$ and let
$\G\car X$ be a free, irreducible, aperiodic action. Suppose that $\La\car Y$ is mildly mixing action. If the action $\G\car X$ is orbit
equivalent with $\La\car Y$ then the two actions are conjugate.
\end{thm*}

Consequently, when this theorem is combined with Corollary~\ref{cor:uniquegmscartan} above, we obtain the following $W^*$-strong rigidity
result:
\begin{cor}\label{strongrig} Let $\G=\G_1\times\G_2$ with $\G_i\in\mathcal{CR}\cap \mathcal C_{\text{reg}}$ and let
$\G\car X$ be a free, irreducible, aperiodic action. Suppose that $\La\car Y$ is mildly mixing action. If the action $\G\car X$ is
$W^*$-equivalent with $\La\car Y$ then the two actions are conjugate.
\end{cor}

Even though it is clear that the classes $\mathcal C_{\text{reg}}$ and $\mathcal{CR}$ do not coincide there is still a considerable overlap
between them. Indeed, combining the examples discussed above with the examples presented in the introduction, the intersection
$\mathcal{CR}\cap \mathcal C_{\text{reg}}$ contains all non-elementary amalgamated free products $\G_1\ast_{\Omega}\G_2$ which satisfy the
following three properties:

\begin{itemize}
  \item $\G_i$ are infinite groups and the common subgroup $\Omega$ is almost malnormal in one of the factors;
  \item  $\G_1$ contains an infinite subgroup which has property (T);
  \item  $\b^{(2)}_1(\G_1)+ \b^{(2)}_1(\G_2)+\frac{1}{|\Omega|}>\b^{(2)}_1(\Omega)$.
\end{itemize}
In particular, if $\G_1$ is an infinite group with property (T), then any non-elementary free product $\G_1\ast\G_2$ belongs to the class
$\mathcal{CR}\cap  \mathcal C_{\text{reg}}$.

\section{Other unique group-measure space decomposition results}\label{sec: other unique decomp}

In this section we will consider a class of groups larger than $\mathcal{CR}$ by not requiring cocycles to be in mixing representations but
rather in representations which are mixing relative to an amenable subgroup. Given a group $\G$ with a subgroup $\Sigma < \G$, and a
representation $\pi: \G \to \UU(\HH)$, we say that $\pi$ is mixing relative to $\Sigma$ if $\langle \pi(\g_n) \xi, \eta \rangle \to 0$ whenever
$\g_n \to \infty$ relative to $\Sigma$. One says that a group $\G$ belongs to $\mathcal{ACR}$ if it satisfies either condition (1) or (2)
below:
\begin{enumerate}

\item\begin{enumerate}
\item\label{item:cocycle1} There exists an amenable normal subgroup $\Sigma \vartriangleleft \G$ and a representation $\pi: \G \to \UU(\HH)$
 which is mixing relative to $\Sigma$; There exists an unbounded cocycle  $c:\G \to \HH$
 which vanishes on $\Sigma$;
\item\label{item:rigid1}  There exists a non-amenable subgroup $\Omega < \G$ such that the pair $(\G,\Omega)$ has relative property (T).
\end{enumerate}

\item\begin{enumerate}
\item\label{item:cocycle2} There exists an amenable normal subgroup $\Sigma \vartriangleleft \G$ and a representation $\pi: \G \to \UU(\HH)$
 which is mixing relative to $\Sigma$ and weakly contained in the left regular; There exists an unbounded cocycle  $c:\G \to \HH$
 which vanishes on $\Sigma$;
\item\label{item:rigid2}  There exists a non-amenable subgroup $\Omega < \G$ which is a product of two nonamenable groups.
\end{enumerate}

\end{enumerate}

Notice that any nonamenable group with positive first $\ell^2$-Betti number admits an unbounded cocycle into the left regular representation.
Therefore $\mathcal{ACR}$ contains every group $\G$ with first positive $\ell^2$-Betti number that admits a nonamenable subgroup $\Omega$ of
$\G$, such that the pair $(\G, \Omega)$ has relative property (T) or $\Omega$ is a product of two nonamenable groups.

Natural classes of groups in $\mathcal{ACR}$ have already been considered in \cite{popavaes09} and \cite{fimavaes2010}.

\begin{examp}If $\G_1, \G_2$ are two groups which contain a common finite subgroup $\Sigma$, and $\G = \G_1 *_\Sigma \G_2$ then we may consider the cocycle
$c: \G \to \ell^2(\G/\Sigma)$ which satisfies $c(\g_1) = \dd_\Sigma - \l(\g_1) \dd_\Sigma$ for $\g_1 \in \G_1$ and $c(\g_2) = 0$ for $\g_2 \in
\G_2$.  By the universal property of amalgamated free products this cocycle extends to all of $\G$ and will be unbounded as long as $\Sigma
\not= \G_1, \G_2$ (if $\g_i \in \G_i \setminus \Sigma$, then it is easy to see that $c$ is unbounded on $\{ (\g_1\g_2)^n \ | \ n \in \N \}$).

This shows that $\G$ satisfies condition (\ref{item:cocycle1}) or (\ref{item:cocycle2})  above.  Finding examples of this type which also
satisfy (\ref{item:rigid1}) or (\ref{item:rigid2}) is not difficult and we refer the reader to \cite{popavaes09} for examples.
\end{examp}

\begin{examp}
Suppose $H$ is a group which contains a finite subgroup $\Sigma$, and $\theta:\Sigma \to H$ is an injective homomorphism.  Let $\G = {\rm
HNN}(H, \Sigma, \theta) = \langle H, t \ | \ \theta(\sigma) = t \sigma t^{-1}, {\rm \ for \ all \ } \sigma \in \Sigma \rangle$ denote the HNN
extension then consider the cocycle $c:\G \to \ell^2(\G/\Sigma)$ given by $c(h) = 0$, for all $h \in H$, and $c(t) = \l(t) \dd_\Sigma$. We
leave it as an exercise to show that it extends to a well defined cocycle on $\G$, and we have that $\| c(t^n) \|^2 = |n|$ for all $n \in \Z$,
hence $c$ is unbounded.  In fact, for this example and the previous one, the cocycle we consider is well known, and arrises naturally from the
action of $\G$ on its Bass-Serre tree.

Then we have that $\G$ satisfies condition (\ref{item:cocycle1}) or (\ref{item:cocycle2})  above.  Again, finding examples of this type which
also satisfy (\ref{item:rigid1}) or (\ref{item:rigid2}) is not difficult and we refer the reader to \cite{popavaes09} for examples.

\end{examp}

The main theorem we prove in this section is an intertwining result (Theorem \ref{uniquecartan3} below) for trace preserving actions of groups
in class $\mathcal{ACR}$ on amenable algebras. As a consequence we obtain new von Neumann algebras with unique group measure space Cartan
subalgebras.

Our proof follows the same general strategy used to prove Theorem \ref{mainintertwining} above and it will be rather sketchy. First, by similar
arguments as in the proof of Lemma \ref{transfer} and Lemma 3.1 in \cite{popavaes09}, we show that a transfer lemma still holds for von Neumann
algebras associated with actions of groups belonging to $\mathcal{ACR}$.
\begin{lem}\label{transfer2}Let $\G\in \mathcal{ACR}$ and let $\G \car A$ be a trace preserving action on an amenable von Neumann algebra.  Suppose also that $B$
is an abelian von Neumann algebra and $\La \car B$ a free, ergodic action such that $N = A \rtimes \G = B \rtimes \La$. If $P \subset N$ is any
amenable von Neumann subalgebra then for every $\e
>0$, there exists $\alpha_{\e}>0$, and an infinite set of elements $S\subset \La$, such that
\begin{eqnarray}
\label{666}&\|\dd_{\alpha}(v_{\l})\|&\le \e \text{ for all } \l\in S, \alpha>\alpha_{\e} \\
\label{566} &\|E_{P}(xv_{\l}y)\|_2& \ra 0 \text{ for all } x,y\in N \text{ as }\l\ra\infty.
\end{eqnarray}
\end{lem}
\begin{proof}
Fix $\e>0$. Borrowing the same notations from the proof of Lemma \ref{transfer} we briefly argue that in both cases there exists
$\alpha_{\e}>0$ such that for all $\alpha>\alpha_{\e}$ we have
\begin{equation}\label{166}
\| \hat \dd_\alpha \circ \Delta(x) \| \le \frac{\e}{2}, {\rm \ for \ all \ } x \in (L\Omega)_1.
\end{equation}

When the pair $(\G, \Omega)$ has relative property (T), this follows from the proof of Lemma \ref{transfer}, so it only remains to prove it
when $\Omega$ is a product of nonamenable groups. This case however follows by referencing the same proof from Theorem 4.3 in
\cite{peterson2006} and using Lemma \ref{lem:diagonalbimodule}. We leave the details to the reader.

Next we fix $w\in L\Omega$ a unitary and let $w=\sum_{\l\in \La} w_{\l}v_{\l}$ with $w_{\l}\in B$ be its Fourier expansion in $B\rtimes \La=N$.
If we denote by $S=\{\l\in \La \ | \ \|\dd_{\alpha} (v_{\l})\|_2\le \e \}$, then using (\ref{166}) in the same manner as in the proof of Lemma
\ref{transfer} we obtain that for all $w\in \mathcal U(L\Omega)$ we have
\begin{equation}\label{366}
\sum_{\l\in S}\|w_{\l}\|^2_2\ge \frac{3}{4}.
\end{equation}

So to finish the proof it only remains to check (\ref{566}) for the set $S$. However this follows from using relation (\ref{366}) above exactly
as shown in the last part in the proof of Lemma 3.2 in \cite{popavaes09}.\end{proof}

Pairing the above transfer lemma with the criterion for uniform convergence of the resolvent deformation described in Section
\ref{sec:convergence} we obtain the following:

\begin{thm}\label{uniquecartan3}
Suppose $\G \in \mathcal{ACR}$, $A$ is an abelian von Neumann algebra and $\G \car A$ is a free, ergodic action.  If $B$ is an abelian von
Neumann algebra and $\La \car B$ is a free action such that  $N = A \rtimes \G = B \rtimes \La$ then $B\prec_NA\rtimes \Sigma$. Moreover if
$\Sigma=\{e\}$ there exists a unitary $u \in \UU(N)$ such that $uBu^* = A$.
\end{thm}

\begin{proof} Since $\Sigma$ is amenable it follows that the von Neumann algebra $A\rtimes \Sigma$ is also amenable. Applying
Lemma \ref{transfer2} for $P=A\rtimes \Sigma$ we obtain that for every $k\in\mathbb N$ there exists $\alpha_k>0$ and an infinite sequence
$\{v_{\l^k_n} \ | \ n\in \mathbb N\}\subset \La$ satisfying the following:
\begin{eqnarray*}
&&\|\dd_{\alpha}(v_{\l^k_n})\|\le \frac{1}{k}\text{ for all } \alpha>\alpha_k, k, n \in \mathbb N; \\
&&\|E_{A\rtimes \Sigma}(xv_{\l^k_n}y)\|_2\ra 0 \text{ as }n\ra\infty,\text{ for each }k\in \mathbb N.
\end{eqnarray*}

Since $\G\in \mathcal{ACR}$ then it admits an unbounded 1-cocycle into a representation which is mixing relative to $\Sigma$ and hence Theorem
\ref{derivunifconv} implies that $\dd_\alpha$ converges uniformly on $(B)_1$.

Next we proceed by contradiction to show $B\prec_N A\rtimes \Sigma$. Assuming $B\nprec_N A\rtimes \Sigma$, by Popa's intertwining techniques
\cite{popa2004}, there exists a sequence of unitaries $b_n\in \mathcal{U}(B)$ such that $\|E_A(xb_ny)\|_2\ra 0$ as $n\ra \infty$. Therefore
Theorem 4.3 in \cite{peterson2006} implies that $\dd_\alpha$ converges to zero uniformly on the unit ball of $\mathcal N_N(B)''=N$ which is
obviously a contradiction.

When $\Sigma=\{e\}$ we have $B\prec_N A$ and by Theorem \ref{intertwining-conj} there exists a unitary $u \in \UU(N)$ such that $uBu^* =
A$.\end{proof}
\begin{rem} If one can remove the normality assumption on the subgroup $\Sigma<\G$ in the proof of Lemma \ref{orthog} then the previous
intertwining result holds for any $\Sigma$ and therefore will allow one to completely recover the $W^*$-superrigidity results obtained in
\cite{popavaes09,fimavaes2010}.
\end{rem}
\begin{rem} Taking $\Sigma=\{e\}$ in the previous theorem we have that \emph{any} free, ergodic action
on a probability space $\G\curvearrowright X$ of a group $\G$ with positive first $\ell^2$-Betti number that admit a nonamenable subgroup
$\Omega <\G$ with relative property (T), gives rise to a von Neumann algebra with unique group measure space Cartan subalgebra. This result may
be interpreted as a positive evidence supporting the following general conjecture of Ioana, Popa and the authors:\end{rem}
 \begin{conj}\emph{Any} free, ergodic action on a probability space $\G\curvearrowright X$ of \emph{any} group $\G$ with positive
first $\ell^2$-Betti number gives rise to a von Neumann algebra with unique Cartan subalgebra.
\end{conj}
Even though a positive answer to this conjecture in its full generality is still out of reach there are instances when it is known to be true.
All known examples however assume some strong conditions on either the group or the action mostly to insure that, besides strong deformability,
the von Neumann algebra $L^{\infty}(X)\rtimes \G$ also possesses a strong pole of rigidity. To enumerate a few examples: \begin{itemize}\item
Any \emph{profinite} action $\G\curvearrowright X$ where $\G$ is a nonamenable free group \cite{ozawapopa07,ozawapopa2010} or more generally
any group with positive first $\ell^2$-Betti number that has the completely metric approximation property \cite{dabrowski,sinclair}. In this
case rigidity arises from the complete metric approximation property and profiniteness of the action; \item Any \emph{profinite} action
$\G\curvearrowright X$ where $\G$ is a group that admits an unbounded $1$-cocycle into a mixing representation and does not have the Haagerup
property \cite{peterson2009}. Here rigidity arises as a mix between profiniteness of the action and the absence of Haagerup property for
$\G$;\item Any action $\G\curvearrowright X$ where $\G$ is any free product of a nontrivial group and either an infinite property (T) group or
a product of two nonamenable groups \cite{popavaes09,fimavaes2010}. Obviously in this case rigidity is inherited from the acting group
$\G$.\end{itemize} Note that our result is mostly in the spirit of the second and third situation above. It will be very interesting to
investigate if the conjecture is true in cases where apriori there is a lack of rigidity, for instance for \emph{any} free, ergodic action of a
nonamenable free group.
\begin{appendix}
\section{On Popa's unique HT Cartan subalgebra Theorem}

We end by mentioning that much of the difficulty in the previous theorems was to obtain uniform convergence of a deformation on the ``mystery''
Cartan subalgebra. If we assume that we already have convergence (for instance if we assume the inclusion $(B \subset N)$ is rigid) then much
of the difficulty diminishes.  In this setting we can weaken our assumptions on the group $\G$ and in this way obtain a generalization of
Popa's unique HT Cartan subalgebra theorem (Theorem 6.2 in \cite{popa2001}).  This result was previously presented by the second author at the
Workshop on von Neumann Algebras and Ergodic Theory held at UCLA in 2007.

\begin{thm}
Let $\G$ be a group such that there exists an unbounded cocycle into a mixing representation.  Suppose $A$ is an abelian von Neumann algebra,
and $\G \car A$ is a free, ergodic action.  Let $N = A \rtimes \G$ and suppose that $(B_0 \subset N)$ is a rigid inclusion such that $B = B_0'
\cap N$ is a Cartan subalgebra.  Then there exists a unitary $u \in \UU(N)$ such that $uBu^* = A$.
\end{thm}
\begin{proof}
Let $\dd:N \to \HH$ be the corresponding derivation associated to $\G$.  Since $\dd_\alpha$ converges uniformly on $B_0$, if $B_0 \not\prec_N
A$ then using Lemma~\ref{14} together with Popa's intertwining theorem \cite{popa2004} and the proof of Theorem 4.5 in \cite{peterson2006} we
have that $\dd_\alpha$ converges uniformly on $(B)_1 \subset (\NN_N(B_0)'')_1$.  The same argument then implies that $\dd_\alpha$ converges
uniformly on $(N)_1 = (\NN_N(B)'')_1$ which would contradict the original cocycle being unbounded.

Thus $B_0 \prec_N A$ and hence by Popa's conjugacy criterion for Cartan subalgebras we obtain the result.
\end{proof}

\begin{rem}
By \cite{akemannwalter} if a group $\G$ has the Haagerup property then there exists a proper cocycle into a mixing representation and so the
previous result applies when $\G$ has the Haagerup property as in Theorem 6.2 in \cite{popa2001}.
\end{rem}

\begin{cor}
Let $\G$ be a group such that $\b_1^{(2)}(\G) > 0$, and suppose $0 < \b_n^{(2)}(\G) < \infty$ for some $n \in \N$. Then there exists a free
ergodic action $\G \car A$ such that $\mathcal F( A \rtimes \G ) = \{ 1 \}$.
\end{cor}

\begin{proof}
By the results in \cite{epstein07}, \cite{gaboriau2010}, \cite{ioana2007}, and \cite{gaboriau2005}, (this appears explicitly as Theorem 4.3 in
\cite{ioana2007}), every non-amenable group $\G$ has a free ergodic action $\G \car A$ such that there is a rigid inclusion $(A_0 \subset N)$
with $A = A_0' \cap N$ (in fact they have uncountably non-orbit equivalent such actions).  If $\b_1^{(2)}(\G) > 0$ then by the previous theorem
$A$ is the only Cartan subalgebra which contains such a rigid subalgebra.  Therefore $\mathcal F(A \rtimes \G) = \mathcal F( A \subset A
\rtimes \G)$, and the result follows from \cite{gaboriaubetti}.
\end{proof}

\end{appendix}


\begin{thebibliography}{{Dab}10}

\bibitem[AW81]{akemannwalter}
Charles~A. Akemann and Martin~E. Walter.
\newblock Unbounded negative definite functions.
\newblock {\em Canad. J. Math.}, 33(4):862--871, 1981.

\bibitem[BV97]{bekkavalette}
Mohammed E.~B. Bekka and Alain Valette.
\newblock Group cohomology, harmonic functions and the first {$L\sp 2$}-{B}etti
  number.
\newblock {\em Potential Anal.}, 6(4):313--326, 1997.

\bibitem[Con76]{connes76}
Alain Connes.
\newblock Classification of injective factors. {C}ases {${\rm II}\sb 1$}
  {${\rm II}\sb{\infty },$} {${\rm III}\sb{\lambda },$} {$\lambda \not=1$}.
\newblock {\em Ann. of Math. (2)}, 104(1):73--115, 1976.

\bibitem[{Dab}10]{dabrowski}
Yoann Dabrowski.
\newblock {A non-commutative Path Space approach to stationary free Stochastic
  Differential Equations}.
\newblock {\em ArXiv e-prints}, June 2010.

\bibitem[DCH85]{decanniere-haagerup}
Jean De~Canni{\`e}re and Uffe Haagerup.
\newblock Multipliers of the {F}ourier algebras of some simple {L}ie groups and
  their discrete subgroups.
\newblock {\em Amer. J. Math.}, 107(2):455--500, 1985.

\bibitem[{Eps}07]{epstein07}
Inessa Epstein.
\newblock {Orbit inequivalent actions of non-amenable groups}.
\newblock {\em ArXiv e-prints}, July 2007.

\bibitem[FV10]{fimavaes2010}
Pierre Fima and Stefaan Vaes.
\newblock {HNN extensions and unique group measure space decomposition of {${\rm II}\sb 1$}
  factors}.
\newblock {\em ArXiv e-prints}, May 2010. To appear in {\em Trans. Amer. Math. Soc.}

\bibitem[Fur99a]{furman99b}
Alex Furman.
\newblock Gromov's measure equivalence and rigidity of higher rank lattices.
\newblock {\em Ann. of Math. (2)}, 150(3):1059--1081, 1999.

\bibitem[Fur99b]{furman99a}
Alex Furman.
\newblock Orbit equivalence rigidity.
\newblock {\em Ann. of Math. (2)}, 150(3):1083--1108, 1999.


\bibitem[Gab02]{gaboriaubetti}
Damien Gaboriau.
\newblock Invariants {$l\sp 2$} de relations d'\'equivalence et de groupes.
\newblock {\em Publ. Math. Inst. Hautes \'Etudes Sci.}, (95):93--150, 2002.

\bibitem[GL09]{gaboriau2010}
Damien Gaboriau and Russell Lyons.
\newblock A measurable-group-theoretic solution to von {N}eumann's problem.
\newblock {\em Invent. Math.}, 177(3):533--540, 2009.

\bibitem[GP05]{gaboriau2005}
Damien Gaboriau and Sorin Popa.
\newblock An uncountable family of nonorbit equivalent actions of {$\Bbb F\sb
  n$}.
\newblock {\em J. Amer. Math. Soc.}, 18(3):547--559 (electronic), 2005.

\bibitem[{HPV}10]{houdayer-popa-vaes-2010}
Cyril Houdayer, Sorin Popa and Stefaan Vaes.
\newblock {Two remarks on the crossed product decomposition of {${\rm II}\sb 1$} factors}.
\newblock {\em ArXiv e-prints}, October 2010.


\bibitem[{Ioa}07]{ioana2007}
Adrian Ioana.
\newblock {Orbit inequivalent actions for groups containing a copy of $\mathbb F_2$}.
\newblock {\em ArXiv e-prints}, December 2007.

\bibitem[Ioa08]{ioana08}
Adrian Ioana.
\newblock Cocycle {S}uperrigidity for {P}rofinite {A}ctions of {P}roperty {(T)}
  {G}roups, 2008.

\bibitem[{Ioa}10]{adi2010}
Adrian Ioana.
\newblock {W*-superrigidity for Bernoulli actions of property (T) groups}.
\newblock {\em ArXiv e-prints}, February 2010.

\bibitem[IPP08]{ioana2005}
Adrian Ioana, Jesse Peterson, and Sorin Popa.
\newblock Amalgamated free products of weakly rigid factors and calculation of
  their symmetry groups.
\newblock {\em Acta Math.}, 200(1):85--153, 2008.
\bibitem[IPV10]{ioana-popa-vaes-2010}
Adrian Ioana, Sorin Popa, and Stefaan Vaes.
\newblock A class of superrigid group von Neumann algebras.
\newblock {\em ArXiv e-prints}, August 2010.
\bibitem[JP82]{jonespopa}
VFR Jones and Sorin Popa.
\newblock Some properties of {MASA}s in factors.
\newblock In {\em Proceedings of VI-th {C}onference in {O}perator {T}heory.},
  pages 210--220. Herculane-Timisoara 1981 I, Gohberg (ed), Birkhauser Verlag,
  1982.

\bibitem[Kid06]{kida2006}
Yoshikata Kida.
\newblock Measure equivalence rigidity of the mapping class group.
\newblock {\em Ann. of Math. (2)}, 171(3):1851--1901, 2010.

\bibitem[Kid09]{kida-2009}
Yoshikata Kida.
\newblock Rigidity of amalgamated free products in measure equivalence theory,
  2009.

\bibitem[MS04]{monodshalom1}
Nicolas Monod and Yehuda Shalom.
\newblock Cocycle superrigidity and bounded cohomology for negatively curved
  spaces.
\newblock {\em J. Differ. Geom.}, 67(3):395--455, 2004.

\bibitem[MS06]{monodshalom}
Nicolas Monod and Yehuda Shalom.
\newblock Orbit equivalence rigidity and bounded cohomology.
\newblock {\em Ann. of Math. (2)}, 164(3):825--878, 2006.

\bibitem[MVN36]{MvN1}
F.~J. Murray and John von Neumann.
\newblock On rings of operators.
\newblock {\em Ann. of Math. (2)}, 37(1):116--229, 1936.

\bibitem[MvN43]{MvN2}
F.~J. Murray and John von Neumann.
\newblock On rings of operators. {IV}.
\newblock {\em Ann. of Math. (2)}, 44:716--808, 1943.

\bibitem[OP07]{ozawapopa07}
Narutaka Ozawa and Sorin Popa.
\newblock On a class of {${\rm II}\sb 1$} factors with at most one {C}artan subalgebra, {I}
\newblock {\em Ann. of Math. (2)}, 172(1):713--749, 2010.

\bibitem[OP10]{ozawapopa2010}
Narutaka Ozawa and Sorin Popa.
\newblock On a class of {${\rm II}\sb 1$} factors with at most one {C}artan
  subalgebra, {II}.
\newblock {\em Amer. J. Math.}, 132(3):841--866, 2010.

\bibitem[Pet09a]{peterson04}
Jesse Peterson.
\newblock A 1-cohomology characterization of property ({$T$}) in von {N}eumann
  algebras.
\newblock {\em Pacific J. Math.}, 243(1):181--199, 2009.

\bibitem[Pet09b]{peterson2006}
Jesse Peterson.
\newblock {$L\sp 2$}-rigidity in von {N}eumann algebras.
\newblock {\em Invent. Math.}, 175(2):417--433, 2009.

\bibitem[Pet09c]{peterson2009}
Jesse Peterson.
\newblock Examples of group actions which are {$W^{*}E$}-superrigid, 2009.

\bibitem[PS09]{petersonsinclair09}
Jesse Peterson and Thomas Sinclair.
\newblock On cocycle superrigidity for gaussian actions, \newblock {\em ArXiv e-prints} 2009.
To appear in {\em  Erg. Theory Dyn. Syst.}

\bibitem[PT07]{petersonthom2007}
Jesse Peterson and Andreas Thom.
\newblock {Group cocycles and the ring of affiliated operators}.
\newblock {\em ArXiv e-prints}, August 2007.

\bibitem[Pop86]{popacorr86}
Sorin Popa.
\newblock {C}orrespondences, 1986.

\bibitem[Pop06a]{popa2001}
Sorin Popa.
\newblock On a class of type {${\rm II}\sb 1$} factors with {B}etti numbers
  invariants.
\newblock {\em Ann. of Math. (2)}, 163(3):809--899, 2006.

\bibitem[Pop06b]{popa2004}
Sorin Popa.
\newblock Strong rigidity of {${\rm II}\sb 1$} factors arising from malleable actions
  of {$w$}-rigid groups. {I}.
\newblock {\em Invent. Math.}, 165(2):369--408, 2006.

\bibitem[Pop07]{popa2006}
Sorin Popa.
\newblock Cocycle and orbit equivalence superrigidity for malleable actions of
  {$w$}-rigid groups.
\newblock {\em Invent. Math.}, 170(2):243--295, 2007.

\bibitem[Pop08]{popa07b}
Sorin Popa.
\newblock On the superrigidity of malleable actions with spectral gap.
\newblock {\em J. Amer. Math. Soc.}, 21(4):981--1000, 2008.

\bibitem[PV09]{popavaes09}
Sorin Popa and Stefaan Vaes.
\newblock Group measure space decomposition of {${\rm II}\sb 1$} factors and {$W^{*}$}-superrigidity,
\newblock {\em Invent. Math.}, 182:371--417, 2010.

\bibitem[Sak09]{sako2010}
Hiroki Sako.
\newblock Measure equivalence rigidity and bi-exactness of groups.
\newblock {\em J. Funct. Anal.}, 257(10):3167--3202, 2009.

\bibitem[Sau90]{sauvageot}
Jean-Luc Sauvageot.
\newblock Quantum {D}irichlet forms, differential calculus and semigroups.
\newblock In {\em Quantum probability and applications, {V} ({H}eidelberg,
  1988)}, volume 1442 of {\em Lecture Notes in Math.}, pages 334--346.
  Springer, Berlin, 1990.

\bibitem[{Sin}10]{sinclair}
Thomas Sinclair.
\newblock {Strong solidity of group factors from lattices in SO(n,1) and
  SU(n,1)}.
\newblock {\em ArXiv e-prints}, September 2010.

\end{thebibliography}
\end{document}